\documentclass{article}

\usepackage[all]{xy}
\usepackage[T1]{fontenc}
\usepackage[english]{babel}
\usepackage[utf8]{inputenc}
\usepackage{setspace}
\usepackage{makeidx}
\usepackage[nottoc]{tocbibind}
\usepackage{lmodern}
\usepackage{type1cm}
\usepackage{titletoc}
\usepackage[bf,small,compact]{titlesec}
\usepackage[fixlanguage]{babelbib}
\usepackage[font=small,format=plain,labelfont=bf,up,textfont=it,up]{caption}
\usepackage[usenames,svgnames,dvipsnames]{xcolor}
\usepackage[a4paper,top=2.5cm,bottom=2.5cm,left=2.5cm,right=2.5cm]{geometry}
\usepackage{hyperref}
\usepackage[square,sort,nonamebreak,comma]{natbib}
\fontsize{60}{62}\usefont{OT1}{cmr}{m}{n}{\selectfont}
\usepackage{nameref}
\usepackage{fancyhdr}
\usepackage[dvipsnames]{xcolor}
\usepackage{xargs} 
\usepackage[colorinlistoftodos,prependcaption,textsize=normalsize]{todonotes}

\usepackage{amsmath}
\usepackage{amssymb}
\usepackage{amsthm}
\usepackage{eucal}
\usepackage{mathrsfs}
\usepackage{skull}
\usepackage{mathtools}
\usepackage{marvosym} 

\usepackage{makeidx}

\usepackage[nottoc]{tocbibind}   
\usepackage[fixlanguage]{babelbib}

\usepackage{graphicx}
\usepackage[svgnames]{xcolor}

\usepackage{setspace}

\renewcommand{\phi}{\varphi}

\newcommand{\img}{\operatorname{img}}

\newcommand{\mc}[1]{\mathcal{#1}}

\newcommand{\bb}[1]{\mathbb{#1}}
\newcommand{\frk}[1]{\mathfrak{#1}}

\newcommand{\Z}{\mathbb{Z}}
\newcommand{\N}{\mathbb{N}}
\newcommand{\C}{\mathbb{C}}
\newcommand{\R}{\mathbb{R}}

\newcommand{\T}{\mathbb{T}}

\newcommand{\Ho}{\mathcal{O}}

\newcommand{\Aut}{\operatorname {Aut} }

\newcommand{\bs}{\backslash}

\newcommand{\ip}[1]{\langle #1 \rangle}

    {\@beginparpenalty\predisplaypenalty
     \@endparpenalty\postdisplaypenalty
     \refstepcounter{equation}%
     \trivlist \item[]\leavevmode
       \hb@xt@\linewidth\bgroup $\m@th
         \displaystyle
         \hskip\mymathindent}%
        {$\hfil 
         \displaywidth\linewidth\hbox{\@eqnnum}%
       \egroup
     \endtrivlist}

\newcommand{\dint}{\ \!\mathrm{d}}
\newcommand{\dext}{\mathrm{d}}

\newcommand{\del}{\partial}
\newcommand{\delbar}{\overline{\partial}}

\newcounter{dummy} \numberwithin{dummy}{section}

\newtheorem{dfn}[dummy]{Definition}
\newtheorem{exa}[dummy]{Example}
\newtheorem{rmk}[dummy]{Remark}
\newtheorem{lem}[dummy]{Lemma}
\newtheorem{prp}[dummy]{Proposition}
\newtheorem{thm}[dummy]{Theorem}
\newtheorem{cor}[dummy]{Corollary}

\newcommand{\SF}[3]{ \mathcal C^\infty_{#3}(#1; \Lambda^{#2}) }

\definecolor{darkred}{RGB}{99,00,00}
\newtheoremstyle{question}{\topsep}{\topsep}{\itshape}{}{\color{darkred}\bfseries}{.}{ }{}
\theoremstyle{question}

\newcommandx{\unsure}[2][1=]{\todo[linecolor=red,backgroundcolor=red!25,bordercolor=red,#1]{#2}}
\newcommandx{\change}[2][1=]{\todo[linecolor=blue,backgroundcolor=blue!25,bordercolor=blue,#1]{#2}}
\newcommandx{\info}[2][1=]{\todo[linecolor=OliveGreen,backgroundcolor=OliveGreen!25,bordercolor=OliveGreen,#1]{#2}}
\newcommandx{\improvement}[2][1=]{\todo[linecolor=Plum,backgroundcolor=Plum!25,bordercolor=Plum,#1]{#2}}
\newcommandx{\thiswillnotshow}[2][1=]{\todo[disable,#1]{#2}}

\newcommand{\BB}{\mathcal{B}}

\newcommand{\GG}{{\mathrm{G}}}

\newcommand{\pp}{\mathfrak{p}}

\newcommand{\CO}{{\mathcal{O}}}

\newcommand{\VV}{{\mathcal{V}}}

\title{The cohomology of left-invariant elliptic involutive structures on compact Lie groups}

\author{Max Reinhold Jahnke\footnote{During the development of this work, the author received funding from CNPq (process 140199/2014-4) and CAPES (PDSE 88881.131905/2016-01).}\\ \small{\href{mailto:jahnke@dm.ufscar.br}{jahnke@dm.ufscar.br}} \\ \small{Department of Mathematics} \\\small{Federal University of São Carlos, SP, Brazil.}}
\date{\today}

\begin{document}

\maketitle

\begin{abstract}
    Inspired by the work of Chevalley and Eilenberg on the de Rham cohomology on compact Lie groups, we prove that, under certain algebraic and topological conditions, the cohomology associated to left-invariant elliptic, and even hypocomplex, involutive structures on compact Lie groups can be computed by using only Lie algebras, thus reducing the analytical problem to a purely algebraic one. The main tool is the Leray spectral sequence that connects the result obtained Chevalley and Eilenberg to a result by Bott on the Dolbeault cohomology of a homogeneous manifold.
\end{abstract}


\section{Introduction}


The initial motivation for this work was an article by Chevalley and Eilenberg \cite{chevalley1948cohomology}, where it is proved that the de Rham cohomology of a compact Lie group can be studied by using only left-invariant forms, thus reducing the study of the cohomology to a purely algebraic problem. This led us to question whether something similar could be done to study the cohomology associated with general involutive structures defined on compact Lie groups. It is clear that the involutive structures must have algebraic properties from the Lie group. Therefore it is natural to restrict our attention to involutive structures that are invariant by the action of the group, these are called left-invariant involutive structures.

Some particular cases of left-invariant involutive structures, as well as involutive structures that are invariant by the group actions in homogeneous spaces, have been studied previously. For example, Bott proved \cite{bott1957homogeneous} that under certain topological conditions, the Doulbeault cohomology of compact homogeneous manifolds can be computed by a purely algebraic method. Pittie \cite{pittie1988dolbeault} not only has shown that Bott's results can be applied to study the complex associated with left-invariant complex involutive structures on compact semisimple Lie groups, but has also classified all such structures. We would also like to point out that interest in left-invariant involutive structures continues to grow, as evidenced by the recent works by \cite{korman2014elliptic}, Araújo \cite{araujo@2019global} and Bor and Jacobowitz \cite{bor2019left}.


As the case of de Rham's cohomology and Doulbeault's cohomology can be considered extreme cases of elliptic structures, we decided to find conditions to guarantee that it is possible to study the cohomology of left-invariant elliptic structures on compact Lie groups using only left-invariant forms, or, in general, if it is possible to reduce the computation of the cohomology to a purely algebraic one.


The following is a summary of the results we obtained. Let $\Omega$ be an orientable manifold and let $\mathcal V \subset \C T \Omega$ be a involutive vector bundle of rank $n$ and let $m=N-n$. For the convenience of the reader, we reviewed the basic definitions of the theory of involutive structures in Section \ref{sec:involutive_structures}. There exists a differential complex associated to the involutive structure $\mathcal V$ denoted by $(\mathcal C^\infty(\Omega; \Lambda^{p,q}), \dext')$ with cohomology spaces denoted by $H^{p,q}_{\mathcal C^\infty}(\Omega; \mathcal V)$. This differential operator can be extended to currents and compactly supported currents. In these cases, the respective cohomologies are denoted by $H^{p,q}_{\mathcal D'}(\Omega; \mathcal V)$ and $H^{p,q}_{\mathcal E'}(\Omega; \mathcal V)$.


Let $G$ be a compact Lie group and let $\frk g$ be the complexification of its Lie algebra. There is an one-to-one correspondence between complex Lie subalgebras $\frk h \subset \frk g$ and left-invariant involutive structures. We always denote the left-invariant involutive structures by corresponding Lie algebra $\frk h$. In Section \ref{sec:basicdef}, we showed how to construct many examples of Lie subalgebras having properties that are interesting from the point of view of the theory of involutive structures.

Since we are considering involutive structures defined by an algebra $\frk h$, the cohomology of the associated differential operator is denoted by $H^{p,q}(G; \frk h)$. The result by Chevalley and Eilenberg that inspired our work states that in the de Rham case (when $\frk h = \frk g$), each cohomology class in $H^{p,q}(G; \frk h)$ always has a left-invariant representative. 

Chevalley and Eilenberg also introduced what is now called the Chevalley-Eilenberg complex. It is a complex defined using only Lie algebras and it is similar to the usual differential complex of the exterior derivative, but restricted to left-invariant forms. It is straightforward to extend the definition of the Chevalley-Eilenberg complex to consider the complex induced by a given subalgebra. The complex induced by the Lie algebra $\frk h$ is denoted by $C^{p,q}(\frk g; \frk h)$ with cohomology spaces denoted by $H^{p,q}(\frk g; \frk h)$. We have two ways to look into these complexes. In Section \ref{sec:involutive_structures_on_compact_lie_groups} we construct a restriction of the usual analytical complexes and their cohomologies, called left-invariant complex and cohomologies, and in Section \ref{sec:cohomology_of_lie_algebras} we define a purely algebraic version of them.

In Section \ref{sec:ce_tech}, we adapted ideas from Chevalley and Eilenberg's paper and constructed a homomorphism
\begin{equation}
\label{chp:intro:the_master_isomorphism}
\phi: H^{p,q}(\frk g; \frk h) \to H^{p,q}(G; \frk h)
\end{equation}
with $\frk h \subset \frk g$ being any Lie algebra. Chevalley and Eilenberg proved is that, for $\frk h = \frk g$, this homomorphism is actually an isomorphism. In Lemma \ref{inclusion_left_invatiant_cohomology_injective}, we proved this map is always injective. Since the homomorphism is injective, it is clear that the algebraic properties of $\frk h$ and $\frk g$ play a role in the dimension of $H^{p,q}(G; \frk h)$. In other words: there are algebraic obstruction to global solvability. However, we notice that there are known examples of involutive structures, such as the Example \ref{exa:torus_liouvile}, in which it is impossible for the homomorphism \eqref{chp:intro:the_master_isomorphism} to be surjective. We explain this in detail in Example \ref{exa:torus_liouvile}.

Still adapting ideas from Chevalley and Eilenberg, and by using that $H^{m-p,n-q}_{\mathcal D'}(G; \frk h)$ is the topological dual of $H^{p,q}_{\mathcal C^\infty}(G; \frk h)$, we proved the following: 

\begin{thm}
\label{int:li:thm:duality_of_left}
    Let $\frk h$ be any left-invariant involutive structure defined over a compact Lie group $G$. Suppose that the operators
    \begin{equation}
        \label{eq:two_operators_closed_range}
        \mathcal C^{\infty}(G; \Lambda^{p,q-1}) \xrightarrow{\dext'} \mathcal C^{\infty}(G; \Lambda^{p,q}) \xrightarrow{\dext'} \mathcal C^{\infty}(G; \Lambda^{p,q+1})
    \end{equation}
    have closed range. Then, it holds that every cohomology class in degree $(p,q)$ has a left-invariant representative if, and only if, every cohomology class in degree $(m-p, n-q)$ also has a left-invariant representative.
\end{thm}


Helgason \cite{helgasondifferential} have a different proof that the homomorphism \eqref{chp:intro:the_master_isomorphism} is a isomorphism when $\frk h = \frk g$. By adapting his proof, we proved the following result.
\begin{thm}
\label{int:comp_lie:hypo_subelliptic}
 Let $\frk h$ be an left-invariant hypocomplex structure on a compact Lie group $G$. Then every cohomology class $u \in H^{p,0}_{\mathcal C^{\infty}}(G; \frk h)$ has a left-invariant representative. That is, $H^{p,0}_{\mathcal C^{\infty}}(G; \frk h) = H^{p,0}_{L}(G; \frk h).$
\end{thm}
We initially proved this result for elliptic structures, but the same proof works for hypocomplex structures, a class of involutive structures that share many properties with complex structures. We explain hypocomplexity in Section \ref{hypocomplex}. Combining the results mentioned above, namely Theorem \ref{int:li:thm:duality_of_left}, Theorem \ref{int:comp_lie:hypo_subelliptic} and the fact that in this context the operator $\dext'$ has closed range (see Theorem \ref{thm:w} and \cite{cordaro2019top} for details), we have that  \eqref{chp:intro:the_master_isomorphism} is an isomorphism in degree $(m-p,n)$ for all $p$. In other words, we have:
\begin{thm}
\label{int:left_invariance_on_degree}
Let $\frk h$ be a left-invariant hypocomplex structure on a compact Lie group $G$. Suppose that the operator $\mathcal C^{\infty}(G; \Lambda^{p,n-1}) \xrightarrow{\dext'} \mathcal C^{\infty}(G; \Lambda^{p,n})$ has closed range. Then we have $H^{m-p,n}_{\mathcal C^{\infty}}(G; \frk h) = H^{m-p,n}_{L}(G; \frk h).$
\end{thm}

In the general case, the techniques explained in the last paragraph only work in the degree $(p,0)$, but when restricted to a torus, they also work other degrees. To be more precise, we have:
\begin{thm}
\label{int:comp_lie:hypo_subelliptic_torus_2}
 Let $\frk h$ be a left-invariant elliptic structure on a torus $\mathbb{T}$. Then, every cohomology class $u \in H^{p,q}(\T; \frk h)$ has a representative that is left-invariant.
\end{thm}

Notice that this theorem proves that in the case left-invariant elliptic structures on the torus, the homomorphism \eqref{chp:intro:the_master_isomorphism} is surjective.

The results above make it clear that finding conditions so that the operator $\dext' : \mathcal C^\infty(G; \frk h) \to \mathcal C^\infty(G; \frk h)$ has closed range was foundamental for this work. This was studied in details in \cite{cordaro2019top}, one result we need is stated in Theorem \ref{thm:w} and since we are working in a restricted context, that of compact Lie groups with left-invariant hypocomplex structures, we could prove a stronger result regarding the range of $\dext'$. 

\begin{thm}
\label{int:awesome}
Let $G$ be a compact Lie group endowed with a left-invariant hypocomplex structure $\frk h$. Then the operators
$ \dext' : \mathcal C^{\mathcal M}(G; \Lambda^{p,n-1}) \xrightarrow{} \mathcal C^{\mathcal M}(G; \Lambda^{p,n})$
and
$ \dext' : \mathcal D'_{\mathcal M}(G; \Lambda^{p,n-1}) \xrightarrow{} \mathcal D'_{\mathcal M}(G; \Lambda^{p,n})$ have closed range.
\end{thm}
 
Here $\mathcal C^{\mathcal M}(G)$ denotes the space of smooth function belonging to the Gevrey or Denjoy-Carleman class $\mathcal M$. We refer to \cite{fujita1995gevrey} and \cite{dasgupta2016eigenfunction} for an introduction. The $\mathcal C^{\mathcal M}(G)$ can be endowed with a topology of ({\bf{DFS}}) spaces, and its dual, denoted by $\mathcal D'_{\mathcal M}(G)$ has a topology of a ({\bf{FS}}) space.


Another natural question is whether it is possible to combine the techniques from the de Rham case with the  techniques from the Doulbeault case. The strategy was the following: denote by $\frk g_\R$ the real Lie algebra of $G$. We assume that the group $K = \exp(\frk h \cap \frk g_\R)$ is closed (and thus compact) Lie group. This naturally calls for considering a homogeneous manifold, the quotient of the original Lie group by the closed subgroup. This homogeneous manifold inherits a natural complex structure from the original left-invariant elliptic structure.

By using the Leray-Hirsh Theorem, we were able to combine techniques we learned with Pittie, Bott, Chevalley, and Eilenberg to study the original cohomology. This approach using homogeneous spaces is explored in detail in Section \ref{sec:homogeneous}. The theorem that shows how the separate techniques can be combined the following:

\begin{thm}
\label{int:thm:homogeneous:converges_in_E_2}
 Let $G$ be a compact connected Lie group and let $\frk h \subset \frk g$ be an elliptic involutive structure having closed orbits. Let $\frk k = \frk h \cap \overline{\frk h}$ and assume that there exists an ideal
$\frk u \subset \frk h$ such that $\frk k \oplus \frk u = \frk h$. Consider the homogeneous space $\Omega = G / K$, with $K = \exp(\frk k \cap \frk g_\R)$, endowed with the complex structure $\mathcal V = \pi_* \frk h$ induced by the quotient map $ \pi : G \to \Omega$. Then, we have
\begin{equation}
\label{int:eq:decomposition_in_E_2}
 H^{p,q} (G; \frk h) = \sum_{r + s = q} H^{p,r}(\Omega; \mathcal V) \otimes H^s(K; \C).
\end{equation}
\end{thm}

Its first consequence is the following:
\begin{thm}
\label{int:omegaIsRieaman}
    Let $G$ be a compact Lie group endowed with an elliptic subalgebra $\frk h \subset \frk g$ such that $\dim \frk h = \dim \frk g - 1$ with $\frk k_\R = \frk h \cap \frk g_\R$ closed. Let $K$, $\Omega$ and $\mathcal V$ be defined as in Theorem \ref{int:thm:homogeneous:converges_in_E_2}. Then $$ H^{p,q} (G; \frk h) =  H^q(K; \C) + H^{q-1}(K; \C).$$
\end{thm}

Here we found some dimensional conditions so that the Dolbeault cohomology is as simple as possible.

We also could obtain a more general result that, in a sense, generalizes both Chevalley-Eilenberg and Pittie's Theorems. This is the following Theorem:

\begin{thm}
    \label{int:thm:homogeneous:homo_on_simply3}
    Let $G$ be a connected, semisimple and compact Lie group and suppose that $\frk h \subset \frk g$ is an elliptic involutive structure having closed orbits. Let $\frk k = \frk h \cap \overline{\frk h}$ and assume that there exists an ideal $\frk z \subset \frk h$ such that $\frk k \oplus \frk z = \frk h$.
    Consider the homogeneous space $\Omega = G / K$, with $K = \exp_G(\frk k \cap \frk g_\R)$, endowed with the complex structure $\pi_* \frk h$. With $\frk u_*$ defined as in Bott's theorem (Theorem \ref{thm:bott}), we have $$ H^{p,q}(G; \frk h) \cong \sum_{r + s = q} H^r(\frk u_*, \frk k, \Lambda^p(\frk g / \frk u_*)) \otimes H^s(\frk k).$$
\end{thm}
In this last result, the object $H^r(\frk u_*, \frk k, \Lambda^p(\frk g / \frk u_*))$ is the cohomology space of a Lie algebra complex that can be computed by using purely algebraic methods involving finite dimensional subalgebras of $\mathfrak g$.  Notice that we do not require dimensional restrictions on $\frk h$, but we have to assume that $G$ is semisimple.

\section{Involutive structures}
\label{sec:involutive_structures}

In this section, we briefly recall some of the basic definitions regarding involutive structures. The main concepts introduced here are the concept of involutive structure, the associated differential complexes and their cohomology spaces, which are the main objects of study in this work.

\subsection{The associated differential complex}

Let $\Omega$ be a smooth and orientable manifold of dimension $N$. An involutive structure on $\Omega$ is a smooth subbundle $\mathcal V$ of the complexified tangent bundle $\C T \Omega$ of $\Omega$ such that the Lie bracket of any two smooth local sections of $\mathcal V$ is again a smooth section of $\mathcal V$. We denote the rank of $\mathcal V$ by $n$ and we denote by $\Lambda^k$ the bundle $\Lambda^k \C T^* \Omega$. If $W$ is a smooth vector bundle, we denote by $\mathcal C^\infty(\Omega; W)$ the space of sections of $W$ with smooth coefficients. 

For each $0 \leq p \leq m$ and each open set $U \subset \Omega$ there is a natural $\dext'$ differential complex associated to each involutive structure $\mathcal{V}$. The complex is denoted by
\begin{equation}
\label{h_complex}
 \mathcal C^\infty(U; \Lambda^{p,0}) \xrightarrow{\dext'} \mathcal C^\infty(U; \Lambda^{p,1})  \xrightarrow{\dext'} ... \xrightarrow{\dext'} \mathcal C^\infty(U; \Lambda^{p,q}) \xrightarrow{\dext'} \mathcal C^\infty(U; \Lambda^{p,q+1}) \xrightarrow{\dext'} \ldots
\end{equation} and its cohomology spaces is denoted by
 $H^{p,q}_{\mathcal C^\infty} (U; \mathcal V).$ We refer to \cite{cordaro2019top}, \cite{berhanu2008introduction}, and \cite{treves1992hypo} for a detailed construction of this complex.

We can also construct a complex similar to \eqref{h_complex}
where now the coefficients are distributions over $U$. In this case, we denote the cohomology space by
$ H^{p,q}_{\mathcal D'} (U; \mathcal V).$ In general, if $\mathcal F(U)$  is a subspace of $\mathcal D'(U)$ closed by differentiation, we denote by $\mathcal F(U; \Lambda^{p,q})$ 
the space of sections of $\Lambda^{p,q}$ with coefficients in $\mathcal F$ and we denote the cohomology of $\dext'$ with coefficients in $\mathcal F(U)$ by $ H^{p,q}_{\mathcal F} (U; \mathcal V).$
Let $m=N-n$. By \cite[Propositions VIII.1.2 and VIII.1.3]{treves1992hypo}, there is a natural bracket which turns the spaces
$ \mathcal C_c^\infty(U;\Lambda^{p,q})$ and $\mathcal D'(U;\Lambda^{m-p,n-q})$ (resp. $\mathcal C^\infty(U;\Lambda^{p,q})$ and $\mathcal{E}'(U;\Lambda^{m-p,n-q}))$
into the dual of one another and in such a way that the transpose of $\dext'$ is also $\dext'$, up to a sign.

We can endow each space $\mathcal C^\infty(U;\Lambda^{p,q})$ with a locally convex structure of an ({\bf{FS}}) space. Its dual $\mathcal{E}'(U;\Lambda^{m-p,n-q})$ is then a ({\bf{DFS}}) space and a sequence of definition for its topology can be taken by the  sequence
$\GG_j(U;\Lambda^{m-p,n-q}) := \{u\in H^{-j}_{\tiny{\mbox{loc}}}(U,\Lambda^{m-p,n-q}): \,\mbox{supp}\, u \subset K_j \}, $
where $\{K_j\}$ is an exhaustion of $U$ by compact sets, with the inclusions $\GG_j(U;\Lambda^{m-p,m-q}) \hookrightarrow \GG_{j+1}(U;\Lambda^{m-p,m-q})$ being compact by the Rellich lemma.


\subsection{Hypocomplex structures}
\label{hypocomplex}
    From now on, we assume the stronger property that the structure $\mathcal V$ is {\sl locally integrable}. This means that, in a neighborhood of an arbitrary point $\pp\in \Omega$, there are defined $m=N-n$ smooth functions whose differentials span
    $\mathcal V^\perp$ at each point in a neighborhood of $\pp$. Notice that each of these functions is annihilated by the operator $\dext'$, that is, they are solutions for $\VV$.
    
    In this section, we introduce the concept of hypocomplex structures, a class of structures with some good analytical properties, similar to the class of complex structures. Such analytical properties are going to
    be used to find conditions so that the operator $\dext'$ in maximum degree has closed range.
    
    Let $\pp \in \Omega$. The ring of germs of solutions of $\mathcal V$ at $\pp$ is     $\mathcal S(\pp) = \{ f \in \mathcal C^\infty(\pp) : \dext' f = 0 \},$
    and we denote by $\mathcal O^m_0$ the ring of germs holomorphic functions defined in a neighborhood of $0$ in $\C^m$. Now, consider $Z_1, \ldots, Z_m$ solutions for $\VV$, defined in an open neighborhood of $\pp$ with linearly independent differentials, and denote by $Z$ the map $Z = (Z_1, \ldots, Z_m)$ defined in a neighborhood of $\pp$. Of course, we assume that $Z(\pp)=0$.
    
    We define a ring homomorphism
    $ \lambda : h\in \mathcal O^m_0 \mapsto  h \circ Z \in  \mathcal S(\pp). $
    It is not difficult to prove that $\lambda$ is injective. We shall say that $\mathcal V$ is \emph{hypocomplex} at $\pp$ if $\lambda$ is surjective, that is, if for every $f \in \mathcal S(\pp)$ there exist a holomorphic function $h \in \mathcal O^m_0$ such that $f = h \circ Z$. 

\begin{dfn}\label{dfn:contun}
    We say that $\VV$ satisfies the unique continuation property if given $V\subset \Omega$ open and $u\in \mathcal D'(V;\Lambda^{p,0})$ satisfying $\dext' u=0$ in a nonempty connected open subset $\omega$ of $V$ then $u$ vanishes identically in the component of $V$ that contains $\omega$.
\end{dfn}

    The next result follows directly from the definitions:
    
    \begin{prp}\label{prp:hypoxstar}
        If $\VV$ is hypocomplex at each point, then $\VV$ satisfies the unique continuation property.
    \end{prp}
    
    A particular but important class of involutive structures is that of the elliptic structures. Recall that the \emph{characteristic set of} $\VV$ is the set $T^0 = \VV^\perp \cap T^*\Omega$ and
    that the structure is \emph{elliptic} if $T^0=0$. Every elliptic structure is hypocomplex \cite[Proposition III.5.1]{treves1992hypo} and so
    satisfies the unique continuation property.
    
    Another important class of hypocomplex structures is the one described in the following. Assume that $\Omega$ is endowed with a locally integrable structure $\VV$. Given $(\pp,\xi)\in T^0$, $\xi\neq 0$, the \emph{Levi form} at  $(\pp,\xi) \in T^0$ is the Hermitian form on $\mathcal V_x$ defined by $\mc L_{(\pp,\xi)}(v,w) = (1/2i)\xi([L,\overline{M}]_\pp), ~v,w \in \mathcal V_\pp,$
    in which $L$ and $M$ are any smooth local sections of $\mathcal V$ in a neighborhood of $\pp$ such that $L_\pp = v$ and $M_\pp = w$. We shall say that $\mc L$ is \emph{nondegenerate} if given any point $(\pp,\xi) \in T^0$, with $ \xi \neq 0$, the Hermitian form $\mc L_{(\pp,\xi)}$ is nondegenerate.
    
    The following result is due to Baouendi-Chang-Treves \cite{baouendi1983microlocal}:

    \begin{thm}\label{thm:BCT}
        Let $\Omega$ be endowed with a locally integrable structure for which the Levi form at each $(\pp,\xi)\in T^0$, $\xi\neq 0$, has one positive and one negative eigenvalue. Then $\VV$ is hypocomplex.
    \end{thm}

    We assume now that $\Omega$ is a compact, real-analytic, orientable and connected manifold. In such a situation $\mathcal C^\omega(\Omega)$ can be endowed with a natural  topology of ({\bf DFS}) space. This topology $\mathfrak{T}$ can be described as follows: if we embed $\Omega\subset {\cal N}$ as a maximally real submanifold of a Stein manifold ${\cal N}$ then $({\cal C}^\omega(\Omega),\mathfrak{T})$ is the inductive limit of a sequence $\CO_\infty({\cal U}_\nu)$, where ${\cal U}_\nu \searrow \Omega$ is a fundamental system of neighborhoods of
    $\Omega$ in ${\cal N}$ and the symbol $\infty$ stands for bounded holomorphic functions. 
    
    If we denote by $\BB$ the sheaf of hyperfunctions on $\Omega$ the space $\BB(\Omega)$ can be identified with the topological dual of $(\mathcal C^\omega(\Omega),\mathfrak{T})$
    or else with space of all analytic functionals on ${\cal N}$ carried by $\Omega$. 
    
    The reason why we introduced the concept of hypocomplexity is because the following result of \cite{cordaro2019top} is necessary for one of our main results in this work.

    \begin{thm}\label{thm:w}
        Let $\Omega$ be a compact, real-analytic manifold endowed with a real-analytic hypocomplex structure $\mathcal V$. Then $\dext' : \BB (\Omega; \Lambda^{p,n-1}) \to  \BB (\Omega; \Lambda^{p,n})$ has closed image of finite codimension.
    \end{thm}


\section{Involutive structures on compact Lie groups}
\label{sec:involutive_structures_on_compact_lie_groups}

    
    In this section, we restrict our attention to involutive structures on compact Lie groups. We introduce the concept of left-invariant involutive structure, which is an involutive structure that encodes some algebraic properties of the Lie group and we show that there are many left-invariant structures with interesting analytical properties.
    
    We want to show that the study of the cohomology of some of these left-invariant involutive structures can be done using only the Lie algebras. In order to explain exactly how to do that, we need to introduce the concepts of left-invariant cohomology of Lie algebras and left-invaraint cohomology of Lie algebras induced by subalgebras.
    
    Then we explain how to use some of the Chevalley and Eilenberg's techniques to study the de Rham cohomology. These techniques can be partially adapted to other left-invariant involutive structures and with them we show that it is possible to include the left-invariant cohomology relative to a subalgebra into the usual cohomology. This inclusion shows that there are algebraic obstructions to solvability.
    
    Finally, we use Hodge decomposition and its relation to Lie derivatives to obtain some new results regarding left-invariant cohomologies, namely Theorem \ref{int:comp_lie:hypo_subelliptic} and Theorem \ref{int:comp_lie:hypo_subelliptic_torus_2}. Then, we show one application of Serre duality, namely Theorem \ref{int:left_invariance_on_degree}, and we conclude with the proof of Theorem \ref{int:awesome}, which is the main theorem in this section.

\subsection{Left-invariant involutive structures}
\label{sec:basicdef}

    Let $G$ be a Lie group with a Lie algebra $\frk g_\R$, that is, the set of all left-invariant vector fields of $G$. We denote by $L_x : G \to G$ the left-multiplication, that is, $L_x(g) = xg$ for $g \in G$.

    \begin{dfn}
        A vector bundle $\mathcal V \subset \C T G$ is called \emph{left-invariant} if $(L_x)_*X_g \in \mathcal V_{xg}$ for all $X_g \in \mathcal V_g$.
    \end{dfn}

    By using Lie subalgebras $\frk h \subset \frk g \doteq \frk g_\R \otimes \C$ it is easy to construct \emph{involutive} left-invariant subbundles of $\C T G$. In fact, we just take $\frk h_x = \{X(x) \in \C T_x G: X \in \frk h \}$ and define $\mathcal V_{\frk h} \doteq \bigcup_{x \in G} \frk h_x$. This vector bundle is clearly left-invariant. Clearly, each left-invariant involutive structure $\mathcal V \subset \C T G$, by evaluation at the identity, defines a complex Lie subalgebra $\frk h \subset \frk g$. 
    In other words, on a Lie group $G$, any complex Lie subalgebra $\frk h \subset \frk g$ is in a one-to-one correspondence with left-invariant involutive structures on $G$. We usually denote the vector bundle $\mathcal V_{\frk h}$ by its corresponding Lie algebra $\frk h$. Notice that, since every Lie group is an analytic manifold, every involutive structure over $G$ is a \emph{locally integrable structure} \citep{berhanu2008introduction, treves1992hypo}.

    We are interested in studying properties of the space $H^{p,q}_{\mathcal C^\infty}(G; \frk h)$. From what we know by the local theory, the properties of this space depend on how the Lie algebra $\frk h$ is included on the algebra $\frk g$. However, since we are dealing with global properties on compact Lie groups, we also have to take into account some intrinsic algebraic properties of $\frk h$ and topological aspects of $G$.

    We use the language of Lie algebras to highlight some special locally integrable structures over $G$:

    \begin{dfn}
    \label{dfn:basic_types_of_lie_algebras}
    We shall say that a Lie subalgebra $\frk h \subset \frk g$ defines:
        \begin{itemize}
            \item an \emph{elliptic structure} if $\frk h + \bar {\frk h} = \frk g$. The subalgebra $\frk h$ is called an \emph{elliptic subalgebra};
            \item a \emph{complex structure} if $\frk h \oplus \bar {\frk h} = \frk g$. The subalgebra $\frk h$ is called a \emph{complex subalgebra};
            \item a \emph{Cauchy-Riemann (CR) structure} if $\frk h \cap \bar {\frk h} = 0$. The subalgebra $\frk h$ is called a \emph{CR subalgebra};
            \item an \emph{essentially real structure} if $\frk h = \bar {\frk h}$. The subalgebra $\frk h$ is called a \emph{essentially real subalgebra}.
        \end{itemize}
    \end{dfn}

    For example, when $\frk h = \frk g$, clearly the complex \eqref{h_complex} is the de Rham complex. In this case, notice that we have an elliptic structure that is also an essentially real structure. When $\frk h \oplus \overline{\frk h} = \frk g$, we have a complex structure that is also an elliptic structure over $G$. In this case, the operator $\dext'$ is the $\bar \del$ operator and the associated complex \eqref{h_complex} is the Dolbeault complex. The following example shows how the cohomology and analytical properties of the associated differential complex can change even with small changes in an essentially real Abelian Lie algebra.
    

    \begin{exa}
        \label{exa:torus_liouvile} Consider the two torus $\mathbb{T}^2$ with coordinates $(x,y)$. Its Lie algebra is given by $\frk t = \operatorname{span}_\C \{ \del/\del x,  \del/\del y\} $. Let $\mu$ be a real number, take $L = \del/\del x - \mu \del/\del y$, and define $\frk h = \operatorname{span}_\C \{ L \}$.
        
        It is easy to see that $\mathcal C^\infty(\mathbb{T}^2, \Lambda^{0,1}) \cong \{ f \dext x : f \in \mathcal C^\infty(\mathbb{T}^2)  \} $
        and
        $ \mathcal C^\infty(\mathbb{T}^2, \Lambda^{0,2}) \cong \{0\}.$
        
        By writing $M = \del / \del y$, we have that $L$ and $M$ form a basis for $\frk t$ and the respective dual basis consists of $\zeta = \dext x$ and $\tau = \mu \dext x + \dext y$. Therefore, the operator $\dext$ can be written as
        $\dext u = M(u) \zeta + L(u) \dext x$ and so $\dext'u = (Lu)\dext x$. Thus, we have the following complex
        $$\mathcal  C^\infty(\mathbb{T}^2) \xrightarrow[]{ \dext'} \{ f \dext x : f \in \mathcal C^\infty(\mathbb{T}^2)  \} \xrightarrow[]{ \dext'} \{0\}.$$

        Let $v \in \mathcal C^\infty(\mathbb{T}^2, \Lambda^{0,1})$ and write $v = f \dext x$, in which $f \in \mathcal C^\infty(\mathbb{T}^2)$. Notice that, if we want to show that $v$ is $\dext'$-exact, we have to find a function $u$ satisfying $\dext' u = v$. This is equivalent to solve the partial differential equation $Lu = f$. It is well known that this equation is globally solvable in $\mathcal C^\infty(\mathbb{T}^2)$ if and only if $\mu$
        is not a Liouville number \citep{greenfield1972global} and $\int_{\mathbb{T}^2} f(x,y) \dint x \dint y = 0$.
        
        Thus, if $\mu$ is a non-Liouville number, the dimension of $H^{0,1}(\mathbb{T}^2; \frk h)$ is one. It is also possible to show that, when $\mu$ is rational, the dimension of $H^{0,1}(\mathbb{T}^2; \frk h)$ is infinity and $\dext'$ has closed range. We prove this claim on Lemma \ref{lem:torus_liouville}. When $\mu$ is a Liouville number, we have that $\dext' (C^\infty(\mathbb{T}^2))$ is not closed and, hence the dimension
        $H^{0,1}(\mathbb{T}^2; \frk h)$ also is infinity. We prove this fact on Lemma \ref{lem:torus_liouville_2}.
    \end{exa}

    \begin{lem}
        \label{lem:torus_liouville} If $\frk h$ is defined as in Example \ref{exa:torus_liouvile} and $\mu$ is a rational number, then the dimension of $H^{0,1}(\mathbb{T}^2; \frk h)$ is infinity and the range of the operator $$\dext' : \mathcal  C^\infty(\mathbb{T}^2) \to \{ f \dext x : f \in \mathcal C^\infty(\mathbb{T}^2) \}$$ is closed.
    \end{lem}

    \begin{proof}
        To show that the dimension of $H^{0,1}(\mathbb{T}^2; \frk h)$ is infinity, we only have to show that there exists an arbitrarily large linearly independent set contained in $H^{0,1}(\mathbb{T}^2; \frk h)$. Let $k > 0$ be an integer. Since $\mu$ is rational, the equation $\xi - \mu \eta = 0$ has infinitely many.
        
        Thus, let $(\xi_j, \eta_j) \in \Z \times \Z$, $j=0, \ldots, k$ be a sequence of distinct zeros of the equation $\xi - \mu \eta = 0$. We are going to show that the set $\{e^{i (\xi_j x + \eta_j y)} \dext x \}_{j=0}^k$ is linearly independent. Let $c_j \in \C$ be constants. To show that this
        set is linearly independent, it is enough to show that $\sum_{j=0}^k c_j e^{i (\xi_j x + \eta_j y)}\dext x$ is $\dext'$-exact only when every constant $c_j$ is zero. By using Fourier series, we can easily see that a solution to the equation $Lf = \sum_{j=0}^k c_j e^{i (\xi_j x + \eta y_j)}$ exists if, and only if, every constant $c_j$ is zero.
        
        To prove that the range of the operator $\dext'$ is closed, it is enough to see that, if $f_j$ is in the range of $\dext'$ and is a sequence converging to $f$ in $\mathcal C^\infty(\mathbb{T}^2)$,
        it holds that $\hat f_j(\xi,\eta) = 0$ whenever $\xi - \mu \eta = 0$. Also, since $f_j \to f$ in $\mathcal C^\infty(\mathbb{T}^2)$, we have that
        $\hat f_j(\xi, \eta) \to \hat f(\xi, \eta) $ for all $(\xi, \eta)$ and, thus, $\hat f(\xi, \eta) = 0 $ whenever $\xi - \mu \eta = 0$. This implies that $f$ is in the range of $\dext'$.
        In fact, if we take
        $$u = \sum_{(\xi,\eta) \in \Z^2,~ \xi - \mu \eta \neq 0 } \frac{ \hat f(\xi, \eta)}{ (\xi - \mu \eta) } e^{i (\xi_j x + \eta_j y)},$$ we have that $u$ is smooth, since $\mu$ is rational, and it holds that $Lu = f$.
    \end{proof}

    \begin{lem}
      \label{lem:torus_liouville_2}
      If $\frk h$ is defined as in Example \ref{exa:torus_liouvile} and $\mu$ is a Liouville number, then the range of the operator $$\dext' : \mathcal  C^\infty(\mathbb{T}^2) \xrightarrow[]{ } \{ f \dext x : f \in \mathcal C^\infty(\mathbb{T}^2) \}$$ is not closed.
    \end{lem}

    \begin{proof}
        We again adapt some ideas from \cite{greenfield1972global}. If $\mu$ is a Liouville number, then there exists a sequence $\{(\xi_j, \eta_j)\}_{j \in \N}$ such that $\{\xi_j\}$ and $\{\eta_j\}$ are increasing and
        $|\xi_j - \mu \eta_j| \leq 1/(\xi_j^2 + \eta_j^2)^j$. Let $L$ be as in Example \ref{exa:torus_liouvile}. By defining $f_k = \sum_{j \leq k} \exp(i\xi_jx + i\eta_jy)$, we have that $f_k, Lf_k \in \mathcal C^\infty(\mathbb{T}^2)$ and
        $f = \lim_k f_k \in \mathcal D'(\mathbb{T}^2) \bs \mathcal C^\infty(\mathbb{T}^2)$, but with $Lf \in \mathcal C^\infty(\mathbb{T}^2).$
        Notice that any other solution of $u$ of $Lu = Lf$ is of the form $u = f + c$, with $c \in \C$. Therefore, $Lf$ is not in the range of $\dext' : \mathcal  C^\infty(\mathbb{T}^2) \xrightarrow[]{ } \{ f \dext x : f \in \mathcal C^\infty(\mathbb{T}^2) \}$.
    \end{proof}


\subsubsection{Construction of left-invariant involutive structures}
\label{subsec:existence}

    In Definition \ref{dfn:basic_types_of_lie_algebras}, we introduced some types of Lie algebras by defining certain special involutive structures. Now we construct some examples of such Lie algebras.
    
    Let $G$ be a compact connected Lie group and consider any maximal torus $\mathbb{T} \subset G$. Let $\frk t$ be the complexification of the Lie algebra of $\mathbb{T}$. The Lie algebra $\frk t$ is Abelian and self-normalizing, that is, if $[X,Y] \in \frk t$ for all $X \in \frk t$, then $Y \in  \frk t$. We endow $\frk g_\R$ with an $ad$-invariant inner product, that is, an inner product satisfying
    $\ip{ [X, Y], Z} = -\ip{ Y, [X, Z]},$
    for all $X, Y, Z \in \frk g_\R$.
    Also, we extend it to an Hermitian inner product on $\frk g$, which satisfies
    $\ip{ [X, Y], Z} = -\ip{ Y, [\overline X, Z]}$
    for all $X, Y, Z \in \frk g.$

    Let $T_1, \ldots, T_r$ be a basis for $\frk t_\R$ and consider the operators $\operatorname{ad}_{T_j} : \frk g \to \frk g$, given by $\operatorname{ad}_{T_j}(X) = [T_j, X]$. These operators have the following properties: $\operatorname{ad}_{T_j}$ and $\operatorname{ad}_{T_k}$ commute for every $j$ and $k$ (Jacobi's identity); $\operatorname{ad}_{T_j}$ is skew-Hermitian (and diagonalizable) for all $j$; all $\operatorname{ad}_{T_j}$ share the same eigenspaces.
    
    If $\frk g'$ is one of the eigenspaces, we denote by $\alpha = (\alpha_1, \ldots, \alpha_r)$ the ordered set of eigenvalues associated to $\operatorname{ad}_{T_1}, \ldots, \operatorname{ad}_{T_r}$, specifically, $\operatorname{ad}_{T_j}(X) = \alpha_j X$ for all $X \in \frk g'$. We denote this eigenspace by $\frk g_\alpha$. Each $\alpha$ is called a root of the Lie algebra. We denote by $\Delta$ the set of all roots. Notice that every eigenvalue is purely imaginary.

    By the spectral theorem, we have a decomposition
    $ \frk g = \frk t \oplus \bigoplus_{\alpha \in \Delta} \frk g_\alpha.$
    
    It is sometimes convenient to deal with $\alpha$ as an element of $\frk t^*$,
    in which $\alpha(T_j)$ is defined by $\operatorname{ad}_{T_j}(X) = \alpha(T_j)X$ for $X \in \frk g_\alpha.$

    Suppose that $X \in \frk g_\alpha$ and $Y \in \frk g_\beta$. By Jacobi's identity we have
    $$
    \begin{aligned}
    	\operatorname{ad}_{T_i}([X,Y]) = -[Y,[T_i,X]] - [X,[Y,T_i]] 
        	= -[Y, \alpha(T_i)X] -[X,-\beta(T_i)Y] 
            = \left\{\alpha(T_i) + \beta(T_i) \right\}[X,Y].
    \end{aligned}
    $$
    
    Thus, we conclude that $[X,Y] \in \frk g_{\alpha + \beta}$ if $\alpha + \beta$ is a root and zero otherwise.
    
    Let $\Delta_+ \subset \Delta$ be a maximal set of roots satisfying the following properties: for all $\alpha \in \Delta$, exactly one of $\alpha$ or $-\alpha$ is in $\Delta_+$; if $\alpha, \beta \in \Delta_+$ and $\alpha + \beta$ is a root, then $\alpha + \beta \in \Delta_+$.

    \begin{rmk}
        The dimension of $\frk t$, usually called the \emph{rank} of $G$, is independent of the choice of the maximal torus and $\dim G - \dim_\C \frk t$ always is an even number.
    \end{rmk}

    Now we can define some left-invariant involutive structures.
    
    Since $\frk t$ is Abelian, it is trivial to construct essentially real, elliptic, complex, or CR subalgebras of $\frk t$.
    In fact, take $\frk t_\R$ the real Lie algebra of $\mathbb{T}$ and let $X_1, \ldots, X_r$ be a basis for $\frk t_\R$. We have that $\frk t = \frk t_\R \otimes \C$. We identify $X_j$ with $X_j \otimes 1$. Let $s,t \geq 0$ and $s + t \leq r$. Consider the subalgebra
    $$ \frk u = \operatorname{span}_\C \{ X_s \ldots, X_{s+t}, X_{s+t+1} + i X_{s+it+2}, \ldots X_{r-1} + iX_{r} \}.$$
    Notice that these are involutive structures on the maximal torus.
    
    If $s > 0$ and $t = 0$, we have an essentially real subalgebra; if $s, t > 0$ and $s + t = r$, we have an elliptic subalgebra of $\frk u$; if $r$ is even and $s = 0$, we have a complex subalgebra of $\frk u$; and, finally, if $s = 0$ and $t > 0$, we have a CR structure of $\frk u$.
    
    Now consider $\frk h = \frk u \oplus \bigoplus_{\alpha \in \Delta_+} \frk g_\alpha $. By the preceding discussion,
    $\frk h$ is a Lie subalgebra and is elliptic if $\frk u$ is elliptic, is complex if $\frk u$ is complex, and is CR if $\frk u$ is CR.
    
    \begin{rmk}
        We can remove the hypothesis of $G$ being compact by taking a Cartan subalgebra $\frk t$ and by proving the statement without using the $\operatorname{ad}$-invariant metric.
    \end{rmk}

    Let $\frk h \subset \bigoplus_{\alpha \in \Delta} \frk g_\alpha$ be any elliptic subalgebra of $\bigoplus_{\alpha \in \Delta} \frk g_\alpha$. If $\frk e \subset \frk t$ is an elliptic subalgebra of $\frk t$, then $\frk e \oplus \frk h$ is an elliptic subalgebra of $\frk g$. The most obvious example of such structure is when $\frk h = \bigoplus_{\alpha \in \Delta_+} \frk g_\alpha$. 
    
    Now we construct some concrete examples.
    
    \subsubsection{Involutive structures on \texorpdfstring{$\operatorname{SU}(2)$}{\operatorname{SU}(2)} and \texorpdfstring{$\operatorname{SU}(3)$}{\operatorname{SU}(3)}}
    
    The group $\operatorname{SU}(2)$ is defined as
    $$ \operatorname{SU}(2) \doteq \left\{ \left(  \begin{matrix} z_1 & -\overline{z_2} \\ z_2 & \overline{z_1} \end{matrix} \right) : z_1, z_2 \in \C, ~|z_1|^2 + |z_2|^2 = 1 \right\}.$$ The rank of $\operatorname{SU}(2)$ is 1.
    
    The Lie algebra of $\operatorname{SU}(2)$, denoted by $\frk{su}(2)$ is generated by
    $$ X = \left(  \begin{matrix} 0 & i \\ i & 0 \end{matrix} \right), \quad Y = \left(  \begin{matrix} 0 & -1 \\ 1 & 0 \end{matrix} \right), \quad  T = \left(  \begin{matrix} i & 0 \\ 0 & -i \end{matrix} \right).$$
    
    We have the following relation between $X, Y$ and $T$:
    
    \begin{equation}
    \label{sec:su2:basic_commutator_relations}
    [T, X] = 2Y, \quad [T, Y] = -2X, \quad [X, Y] = 2T.
    \end{equation}

    \begin{exa}
        Let $L = X - iY$ and consider the structure $\frk h = \operatorname{span}_\C \{L\}$. This is clearly a CR structure. We have that $\{L, \overline{L},T\}$ is a basis for $\frk{su}(2) \otimes \C$ and we denote its dual basis by $\{\zeta, \overline{\zeta}, \tau\}$. We note that  $\frk h^\perp_0$ is a vector space of dimension 1 and that $\tau \in \frk h^\perp_0$. Thus, if $\xi \in \frk h^\perp_0$, then we can write $\xi = \lambda \tau$, in which $\lambda \in \R \bs \{0\}$. We note that $\frk h$ also has dimension 1 and so $Z = \alpha L$ for all $Z \in \frk h$. Then, the Levi form for the structure $\frk h$ is given by $$ \mathscr L_\xi (Z, Z) = \frac{1}{2i} \lambda \alpha \overline{\alpha} \tau([L, \overline{L}]) = \frac{1}{2i} \lambda \alpha \overline{\alpha} 4i \tau(T) = 2\lambda|\alpha|^2. $$
        In conclusion, the Levi form is non-degenerate and does not change sign for all $Z \in \frk h$.
    \end{exa}

    We will show that $\operatorname{SU}(2)$ does not admit a Levi flat left-invariant structure. In fact, suppose that $\operatorname{SU}(2)$ admits a Levi flat left-invariant CR structure $\frk h$. Since it is Levi flat, we have that the essentially real structure $\frk h + \overline{\frk h}$ is a Lie algebra. Take $\frk k = \frk{su}(2) \cap (\frk h + \overline{\frk h})$ and notice that $\frk k \otimes \C = \frk h + \overline{\frk h}$.
    
    Next we prove that the algebra $\frk k$ is Abelian. Let $X, Y \in \frk k$ be linearly independent. We have that $\ip{[X,Y],X} = \ip{Y,[X,X]} = 0$ and that $\ip{[X,Y],Y} = -\ip{X,[Y,Y]} = 0$. Therefore, $[X,Y]$ is orthogonal to $\frk k$. Since $\frk k$ is a Lie algebra, we also have that $[X,Y] \in \frk k$ and so the only possibility is that $[X,Y] = 0$. Therefore, there is an Abelian subalgebra of rank 2. This is a contradiction because the rank of
    $\operatorname{SU}(2)$ is 1.
    
    This CR structure is related to the natural complex structure of $\C^2$. There is a natural diffeomorphism between the group $\operatorname{SU}(2)$ and the 3-sphere $\mathbb{S}^3$, which has a natural CR structure inherited from $\C^2$. This structure, when pulled back to $\operatorname{SU}(2)$, is exactly the left-invariant structure we just described. Now we prove this claim:

    The complex structure of $\C^2$ is given by the involutive bundle $\mathcal V = \bigcup_{p\in\C^2} \mathcal V_p$, in which $$\mathcal V_p = \operatorname{span} \left\{ \left. \frac{\del}{\del \overline{z}_j} \right|_p : j = 1, 2 \right\} \text{ and }
     \frac{ \del }{ \del \overline{z}_j } = \frac{1}{2} \left( \frac{ \del }{ \del x_j } + i \frac{ \del }{ \del y_j }\right), \qquad j = 1,2.$$
    
    By writing $\rho(z_1, z_2) = |z_1|^2 + |z_2|^2  - 1$, we have that
    $ \mathbb{S}^3 = \rho^{-1}(0)$. Clearly, $\dext \rho_p \neq 0$ if $p \in \bb S^3 $, so we have
    $ \C T_p \bb S^3 = \{ Z \in \C T_p \C^2 : Z_p (\rho) = 0 \}.$
    
    Notice that $\mathcal W_p = \mathcal V_p \cap \C T_p \bb S^3$ defines a CR structure on $\mathbb S^3$. We want a description of $\mathcal W_p$ to simplify some computations. We note that $ \mathcal W_p = \{ Z \in \mathcal V_p : Z_p (\rho) = 0 \}.$ Notice that if $Z \in \mathcal V_p$, by taking $p_1, p_2 \in \C$ with $p = (p_1, p_2)$, then $$Z = \alpha_1 \left. \frac{\del}{\del \overline{z}_1} \right|_p + \alpha_2 \left. \frac{\del}{\del \overline{z}_2} \right|_p$$ and so $(Z \rho)(p) = 2(a_1p_1 + a_2p_2) = 0.$
    Now we have
    $$ \mathcal W_p = \left\{ \alpha_1 \left. \frac{\del}{\del \overline{z}_1} \right|_p + \alpha_2 \left. \frac{\del}{\del \overline{z}_2} \right|_p
     : \alpha_1, \alpha_2 \in \C;~ a_1p_1 + a_2p_2 = 0 \right\}.$$

    The identification of $\mathbb S ^3$ with $\operatorname{SU}(2)$ is given by the diffeomorphism
    $$  (z_1, z_2) \in \mathbb S ^3 \xmapsto{\theta} \left(  \begin{matrix} z_1 & -\overline{z_2} \\ z_2 & \overline{z_1} \end{matrix} \right) \in \operatorname{SU}(2).$$
    
    This map endows $\mathbb S ^3$ with a structure of Lie group with multiplication given by
    $$ (\alpha_1,\alpha_2) \cdot (z_1,z_2) \doteq \theta^{-1}(\theta(\alpha_1,\alpha_2) \cdot \theta(z_1,z_2)) = (\alpha_1 z_1 - \overline{\alpha_2} z_2, \alpha_2 z_1 + \overline{\alpha_1} z_2).$$
    
    Now, for $\alpha \in \mathbb S^3$, we have the diffeomorphism $ L_\alpha : \mathbb S^3 \to \mathbb S ^3 $ defined by $L_\alpha(z_1, z_2) = (\alpha_1, \alpha_2) \cdot (z_1, z_2)$. Notice that this map has a natural extension to $\C^2$ and each component is a holomorphic function. More than that, for $\alpha \neq 0$, this map is actually a biholomorphism from $\C^2$ to $\C^2$ and so it preserves the complex structure of $\C^2$. Therefore, it clearly restricts to a CR diffeomorphism from
    $\mathbb S^3$ to $\mathbb S ^3$. What we proved is that the involutive structure $\mathcal W$ is invariant by the action of $L_\alpha$, or, in other words, $\mathcal W$ is left-invariant.
    
    Now we have almost everything we need to relate the CR structure we introduced on $\mathbb{S}^3$ with the abstract CR structure we constructed on $\operatorname{SU}(2)$. To simplify even more the exposition, let $p \in \C^2$ and identify $\C T_p \C^2$ with $\C^2$ by
    \begin{equation}
        \label{eq:adele}
        \left. \frac{ \del }{ \del x_1 } \right|_p \mapsto (1,0), \left. \frac{ \del }{ \del y_1 } \right|_p \mapsto (i,0), \left. \frac{ \del }{ \del x_2 } \right|_p \mapsto (0,1) \text{ and } \left. \frac{ \del }{ \del y_2 } \right|_p \mapsto (0,i).
    \end{equation}

    With this identification, we have that $\mathcal W_p$ is identified with $$\{(z_1, z_2) \in \C^2 : z_1p_1 + z_2p_2 = 0 \}.$$ Notice that $ \theta $ is linear and so its differential can be identified with $\theta$ itself. Also, notice that, on $p = (1,0)$, we have that $Z \in \mathcal W_{(1,0)}$ is of the form $Z = \Lambda \del / \del \overline{z_2}$, in which $\Lambda \in \C$. Therefore, all we need to compute is $\theta_*(\del / \del \overline{z_2})$.
    By using the identification \eqref{eq:adele},
    we have
    $$
    \begin{aligned}
    \theta_*(2\del / \del \overline{z_2}) & = \theta_*((0,1) + i(0,i)) 
    		= \left(  \begin{matrix} 0 & -1 \\ 1 & 0 \end{matrix} \right) + i\left(  \begin{matrix} 0 & i \\ i & 0 \end{matrix} \right) 
            = i(X - iY).
    \end{aligned}$$

    \begin{exa}
    \label{exa:there_are_examples}
        Let $L = X - iY$ and consider the structure $\frk h = \operatorname{span}_\C \{L, T\}$. This is clearly a left-invariant elliptic structure.
    \end{exa}


    We recall that $\operatorname{SU}(3)$ is the group of all unitary $3 \times 3$ matrices with complex coefficients having determinant 1. It is a real Lie group, compact, and has dimension 8 and rank 2. The Lie algebra of $\operatorname{SU}(3)$ is denoted by $\frk{su}(3)$ and is the set of all traceless skew-Hermitian $3 \times 3$ matrices with complex coefficients. The following matrices form a basis for $\frk{su}(3)$:
    
    $$
    \begin{aligned}
    T_1 &= \left( \begin{matrix} i & 0 & 0 \\ 0 & -i & 0 \\ 0 & 0 & 0 \end{matrix} \right), & \quad X_1 &= \left( \begin{matrix} 0 & i & 0 \\ i & 0 & 0 \\ 0 & 0 & 0 \end{matrix} \right), & \quad  X_2 &= \left( \begin{matrix} 0 & 0 & i \\ 0 & 0 & 0 \\ i & 0 & 0 \end{matrix} \right), & \quad X_3 &= \left( \begin{matrix} 0 & 0 & 0 \\ 0 & 0 & i \\ 0 & i & 0 \end{matrix} \right), \\ 
    T_2 &= \left( \begin{matrix} i & 0 & 0 \\ 0 & i & 0 \\ 0 & 0 & -2i\end{matrix} \right) , & \quad Y_1 &= \left( \begin{matrix} 0 & -1 & 0 \\ 1 & 0 & 0 \\ 0 & 0 & 0\end{matrix} \right), & \quad Y_2 &= \left( \begin{matrix} 0 & 0 & -1 \\ 0 & 0 & 0 \\ 1 & 0 & 0\end{matrix} \right), & \quad Y_3 &= \left( \begin{matrix} 0 & 0 & 0 \\ 0 & 0 & -1 \\ 0 & 1 & 0\end{matrix} \right).  \\
    \end{aligned}
    $$

We have the following relations between the coefficients. Each cell corresponds to the commutator between the first element of the line and the first element of the column. Since the commutator is anti-symmetric, we omitted half of the commutators.


\bgroup
\def\arraystretch{1.5}%
\begin{center}\begin{tabular}{c || c | c | c | c | c | c | c | c}
         & $T_1$ & $T_2$ & $X_1$  & $Y_1$   & $X_2$       &  $Y_2$       & $X_3$     & $Y_3$   \\
\hline
\hline
$T_1$    & $0$   & $0$   & $2Y_1$ & $-2X_1$ & $Y_2$       &  $-X_2$      & $-Y_3$   & $X_3$       \\
\hline
$T_2$    & -     & $0$   & $0$    & $0$     & $3Y_2$      &  $-3X_2$     & $3Y_3$   & $-3X_3$     \\
\hline
$X_1$    & -     &  -    & 0      & $2T_1$  & $Y_3$       &  $-X_3$      & $Y_2$    & $-X_2$      \\
\hline
$Y_1$    & -     & -     & -      & 0       & $X_3$       &  $Y_3$       & $-X_2$   & $-Y_2$      \\
\hline
$X_2$    & -     & -     & -      & -       & 0           &  $T_2 + T_1$ & $Y_1$    & $X_1$       \\
\hline
$Y_2$    & -     & -     & -      & -       & -           & 0            & $-X_1$   & $Y_1$       \\
\hline
$X_3$    & -     & -     & -      & -       & -           & -            & 0        & $T_2 - T_1 $\\
\hline
$Y_3$    & -     & -     & -      & -       & -           & -            & -        & 0           \\

\end{tabular}\end{center}
\egroup

    \begin{exa}
        Considering the commutators we just computed, we have an obvious CR structure. We define the following vector fields on $\frk{su}(3) \otimes \C$:
        $$L_1 = X_1 - iY_1, \quad L_2 = X_2 - iY_2, \quad L_3 = X_3 - iY_3 .$$
        
        Also, from the commutators, it follows that, for each $j,k$, there exists a real number $\Lambda_{jk}$ such that $[T_j,X_k] = \Lambda_{jk} Y_k$ and $[T_j,Y_k] = -\Lambda_{jk} X_k$. By combining these last two equations,
        we have $ [T_j, L_k] = -i\Lambda_{jk} L_k $, which means that each $L_k$ is an eigenvector for the map $\operatorname{ad}_{T_j}$.
        
        Notice that $[L_1,L_2] = 0$, $[L_1, L_3] = 2iL_2$, and $[L_2,L_3] = 0$. Thus, $$\frk h = \operatorname{span}_\C \{L_1, L_2, L_3\}$$ is a Lie subalgebra of $\frk{su}(3) \otimes \C$, which is clearly a CR subalgebra.
    \end{exa}

    Now we use the structure just defined to construct a hypocomplex and an elliptic one.

    \begin{exa}
        Let $\frk h' = \operatorname{span}_\C\{L_1, L_2, L_3, aT_1 + bT_2\}$. Notice that $L_1, L_2, L_3$, and their conjugates are eigenvectors for $\operatorname{ad}_{T_j}$ and that $[aT_1 + bT_2, aT_1 + bT_2] = 0$. Therefore, this is a Lie subalgebra of $\frk{su}(3) \otimes \C$.
    \end{exa}

    We claim that this structure is hypocomplex. A generic element of $\frk h'$ is of the form $ Z = \alpha L_1 + \beta L_2 + \gamma L_3 + \delta U$ with $ U = aT_1 + bT_2$ and $\alpha, \beta, \gamma, \delta$ complex numbers. Therefore, we have that $$[Z, \overline{Z}] = |\alpha|^2 4i T_1 + |\beta|^2 2i(T_2 + T_1) + |\gamma|^2 2i(T_2 - T_1) + W,$$ in which $W$ is a linear combination of $X_j, Y_j$, for $j = 1, 2, 3$.
    
    For any $\theta \in \frk h^\perp_0$, we can write $ \theta = t(-b \tau_1 + a \tau_2)$, with $t \in \R$ and $t \neq 0$. To simplify the proof, we are take $t=1$. Hence, the Levi form of $\frk h'$ at $(e,\theta)$ can be computed as:
    $$
    \begin{aligned}
    \mathscr L_{\theta} (Z,Z) & = \frac{1}{2i}\theta([Z, \overline{Z}]) \\
    	& =  -2|\alpha|^2 b + |\beta|^2(a - b) + |\gamma|^2(a+b)  \\
        & = b(|\gamma|^2 - 2|\alpha|^2 - |\beta|^2) + a(|\beta|^2 + |\gamma|^2).
    \end{aligned}
    $$
    
    From this expression, if $b \neq 0$, we can easily see that, for any $\theta \in \frk h^\perp_0$, it is possible to find a $Z$ such that $\mathscr L_{\theta} (Z,Z) < 0$.
    Therefore, by Theorem \ref{thm:BCT}, we conclude that this structure is hypocomplex.

    \begin{exa}
        Consider the structure $\frk h'' = \operatorname{span}_\C\{L_1, L_2, L_3, T_1, T_2\}$. Notice that $L_1, L_2, L_3$, and their conjugates are eigenvectors for $\operatorname{ad}_{T_j}$ for $j = 1, 2$ and that $[T_1, T_2] = 0$. This structure is clearly elliptic.
    \end{exa}


\subsection{Left-invariant cohomologies}
\label{sec:ce_tech}

We want to show that, under some reasonable conditions, we can compute the cohomology of the complex \eqref{h_complex} by restricting our attention only to left-invariant forms. To accomplish this, we recall some basic definitions and explain some techniques that Claude Chevalley and Samuel Eilenberg used in \cite{chevalley1948cohomology} to study the cohomology of the de Rham complex.

Then we extend some of these techniques to left-invariant cohomologies. Finally, we discuss conditions for usual cohomology spaces to be isomorphic to the left-invariant cohomology spaces.

\subsubsection{Left-invariant de Rham cohomology}

Let $G$ be a compact Lie group and $\SF{G}{p}{}$ be the set of all smooth $p$-forms on $G$. We recall that a vector field $X$ is called left-invariant if, for every $x \in G$, we have
$ (L_x)_* X = X$. Also, a differential form $u \in \SF{G}{p}{})$ is called left-invariant if, for every $x \in G$, we have $ (L_x)^* u = u $. We denote by $\SF{G}{p}{L}$ the set of all left-invariant $p$-forms. Since the exterior derivative commutes with pullbacks, we have that if $u \in \SF{G}{p}{L}$, then $\dext u \in \SF{G}{p+1}{L}$

Considering the notation from Section \ref{sec:involutive_structures},
we denote the $p$-cocycles by $Z^p(G)$, the $p$-coboundaries by $ B^p(G)$, and the $p$-cohomology classes by $H^p(G) = Z^p(G) / B^p(G).$ We also define in the obvious way the left-invariant $p$-cocycles, left-invariant $p$-coboundary, and left-invariant $p$-cohomology classes, denoted, respectively, by $Z^p_L(G)$, $B^p_L(G)$, and $H^p_L(G)$.

Notice that, for each integer $p$, the inclusion $\SF{G}{p}{L} \hookrightarrow \SF{G}{p}{}$ induces a linear mapping $$i_p : H^p_L(G) \to H^p(G).$$

We show in this section that this linear map is an isomorphism. In order to do this, we need some basic tools. We start by defining the averaging operator.


\begin{lem}
\label{averaging_operator_lemma}
 Let $u \in \SF{G}{p}{}$. The differential form
\begin{equation}
\label{averaging_operator}
  Au = \int_G L^*_x u \dint \mu(x)
\end{equation}
has the following properties:
\begin{enumerate}
\item \label{deRhamLeftInvariantCommutative} $\dext (Au) = A(\dext u)$;
\item $Au$ is left-invariant;
\item if $u$ is left-invariant, then $u = Au$.
\end{enumerate}
\end{lem}

    This lemma is just a particular case of a more general result stated and proved at \cite[page 90]{chevalley1948cohomology}. 

    It is easy to show that $i_p$ is injective. Let $[u]$ be a class in $H^p_L(G)$ and suppose that it is 0 in $H^p(G)$, then there exist a smooth form $v$ such that $u = \dext v$. Notice that $u = Au = A\dext v = \dext (A v).$ That is, $[u]$ is 0 in  $H^p_L(G)$. The surjectivity follows from the de Rham's Theorem on the duality between de Rham cohomology and singular homology and the fact that singular homology is left-invariant. This is proved in detail in \cite{chevalley1948cohomology}.

    Therefore, for each $p$, $i_p$ is an isomorphism of vector spaces and so $ H^p(G) \cong H^p_L(G).$ These isomorphisms simplify a lot of computations because they show that we can restrict our attention to left-invariant forms. Also, left-invariant forms are completely determined by their values at the tangent space at the identity. In other words, we reduced the problem to a problem in linear algebra.


\subsubsection{Left-invariant cohomology relative to a subalgebra}
\label{sec:left_invariant_bidegree_cohomology}

We define $N^{0,q}(G; \frk h) = \mathcal C^\infty(G; \Lambda^q)$ and, for $p > 0$, we define
$$N^{p,q}(G; \frk h) = \{ u \in \mathcal C^\infty(G; \Lambda^{p+q}): u(X_1, \ldots, X_{p+q}) = 0 \text{ when } q+1 \text{ arguments are in } \frk h \}.$$

Also, we define $N^{p,q}_{L}(G; \frk h) \doteq N^{p,q}(G; \frk h) \cap \SF{G}{p+q}{L}$ and, since $\dext \SF{G}{p}{L} \subset \SF{G}{p+q+1}{L},$ we can restrict the exterior derivative and obtain
$$ \dext_L : N^{p,q}_{L}(G; \frk h) \to N^{p,q+1}_{L}(G; \frk h).$$
Therefore, $\dext_L$ induces a coboundary operator on the quotient
$$ \SF{G}{p,q}{L} = N^{p,q}_{L}(G; \frk h) / N^{p+1,q-1}_{L}(G; \frk h),$$
which is denoted by $\dext^{p,q}_{\frk h, L}.$ When there is no risk of confusion, we omit the bidegree writing only $\dext_{\frk h, L}$. When the structure $\frk h$ is obvious by the context, we write $\dext'_{L}$.

For $p \geq 0$, we denote the set of the left-invariant $(p,q)$-cocycles elements by $$ Z_{L}^{p,q} (G; \frk h) = \ker \left( \dext'_L : \SF{G}{p,q}{L}) \to \SF{G}{p,q+1}{ L}) \right),$$
the set of left-invariant $(p,q)$-coboundaries by
 $$ B^{p,q}_{L}(G; \frk h) = \img \left( \dext'_L : \SF{G}{p,q-1}{L}) \to \SF{G}{p,q}{L} \right),$$
and the left-invariant $(p,q)$-cohomology classes by
 $$ H^{p,q}_{L}(G; \frk h) = \frac{Z^{p,q}_{L
 }(G; \frk h)}{B^{p,q}_{L}(G; \frk h)}.$$



Denote by $ \pi^{p,q}_{\frk h} : N^{p,q}(G; \frk h) \to \SF{G}{p,q}{}$ and $ \pi^{p,q}_{\frk h, L} : N^{p,q}_{L}(G; \frk h) \to \SF{G}{p,q}{L}$ the quotient maps.

\begin{lem}
\label{lem:op_A_and_pi_commutes}
The operator $A$ defined on Lemma \ref{averaging_operator_lemma} satisfies:
$ A(N^{p,q}(G; \frk h)) \subset N^{p,q}(G; \frk h). $
Thus, there is an unique operator $A^{p,q}_{\frk h}$ such that $A^{p,q}_{\frk h} \circ \pi^{p,q} = \pi^{p,q}_L \circ A$. 
\end{lem}

\begin{proof}
In fact, let $X_1, \ldots, X_{p+q} \in \frk X(G)$ and suppose that $q+1$ of these vector fields are in $\frk h$. We can assume that $X_1, \ldots, X_{q+1} \in \frk h$. It follows that
$$
\begin{aligned}
 Au (X_1, \ldots, X_{p+q}) & = \int_G (L^*_x u) (X_1, \ldots, X_{p+q}) \dint \mu(x) \\
	& = \int_G u ((L_x)_*X_1, \ldots, (L_x)_*X_{p+q}) \mu(x) \\
	& = \int_G u (X_1, \ldots, X_{q+1}, (L_x)_*X_{q+2},\ldots, (L_x)_*X_{p+q}) \mu(x)
	&= 0.
\end{aligned}
$$
The operator $A^{p,q}_{\frk h} : \SF{G}{p,q}{} \to \SF{G}{p,q}{\frk h,L}$ is defined by $A^{p,q}_{\frk h} (\pi^{p,q}(u)) =  \pi_{p,q+1,L}(Au).$ The uniqueness is obvious.
\end{proof}

Notice that the operator $ A^{p,q}_{\frk h} $ satisfies
$ \dext^{p,q}_{\frk h,L} \circ A^{p,q}_{\frk h} = A^{p,q}_{\frk h} \circ \dext^{p,q}_{\frk h}.$
This follows directly from the definition of $A^{p,q}_{\frk h}$ and from the fact that $A$ and $\dext$ commute.


\subsubsection{Necessary and sufficient conditions for the isomorphism}
\label{subsubsection_necessary_sufficient}

We denote by $\phi^{p,q}_{\frk h}$ the quotient map $\phi^{p,q}: Z^{p,q}(G; \frk h) \to H^{p,q}(G; \frk h)$ and by $\phi^{p,q}_{\frk h, L}$ the quotient map $\phi^{p,q}_{\frk h, L}: Z^{p,q}_{L}(G; \frk h) \to H^{p,q}_{L}(G; \frk h).$

Notice that, if $u \in \SF{G}{p,q}{L}$, we can write $ u = \phi^{p,q}_{\frk h, L}(u') $, in which $ u' \in N^{p,q}_{L}(G; \frk h) \subset N^{p,q}(G; \frk h) $. Also, we can define $i(u) = \phi^{p,q}_{\frk h}(u') \in \SF{G}{p,q}{}$ so we have a map $ i : \SF{G}{p,q}{L} \to \SF{G}{p,q}{}.$
Since this map is injective, we identify the set $\SF{G}{p,q}{L} $ with its image
under $i$ and we assume that $\SF{G}{p,q}{L} \subset \SF{G}{p,q}{} $. The map $ i : \SF{G}{p,q}{L} \to \SF{G}{p,q}{}$ satisfies with $\dext_{\frk h} \circ i = i \circ \dext_{\frk h, L}$, so it induces a map
$  i : H^{p,q}_{L}(G; \frk h) \to H^{p,q}(G; \frk h).$

\begin{lem}
\label{inclusion_left_invatiant_cohomology_injective} The map $i : H^{p,q}_{L}(G; \frk h) \to H^{p,q}(G; \frk h)$ is injective.
\end{lem}

\begin{proof}
 Let $u \in H^{p,q}_{L}(G; \frk h)$, with $i(u) = 0$. We write $u = \pi^{p,q}_{\frk h, L}(u')$, in which $u' \in \SF{G}{p,q}{L}$ satisfies $\dext'_L u' = 0$. Since $i(u) = 0$, there exists $v \in \SF{G}{p,q}{}$ with $\dext' v = u$. Thus, we have that $\dext'(Av) = A \dext' v = A u = u$
 and so $u = 0$.
\end{proof}

Since the map $i : H^{p,q}_{L}(G; \frk h) \to H^{p,q}(G; \frk h)$ is always injective, we can already show some algebraic obstructions that can prevent the cohomology space $H^{p,q}(G; \frk h)$ from being zero.
That is, we always have $ \dim H^{p,q}_{L}(G; \frk h) \leq \dim H^{p,q}(G; \frk h)$.

To show that the map $i : H^{p,q}_{L}(G; \frk h) \to H^{p,q}(G; \frk h)$ is surjective, given $v \in \SF{G}{p,q}{}$, we have to find $u \in \SF{G}{p,q}{L}$ such that $v - u = \dext' w$ for some $w \in \SF{G}{p,q-1}{}$.
The natural candidate for $u$ is $Av$. Thus, we want to find a $w \in \SF{G}{p,q-1}{}$ satisfying $Av - v = \dext' w$.
In other words, we want to solve the following problem: given $v \in \ker A^{p,q}_{\frk h}$ with $\dext' v = 0$, can we find $u \in \SF{G}{p,q}{}$ satisfying $\dext' u = v$? If $\frk h$ is a involutive structure such that this problem has a solution in degree $(p,q)$, then we say that it has \textit{property (K) in degree $(p,q)$}.

\begin{exa}
 In the case in which $\frk h = \frk g$, we have that $\dext'$ is the usual exterior derivative. Also, de Rham's theorem implies that this structure always have property (K).
\end{exa}

\begin{lem}
 Suppose that the map $i : H^{p,q}_{\frk h, L}(G) \to H^{p,q}_{\frk h}(G)$ is surjective. Then the involutive structure $\frk h$ has property (K) in degree $(p,q)$.
\end{lem}

\begin{proof}
In fact, let $v \in \SF{G}{p,q}{}$ such that $\dext' v = 0$ and satisfying $Av = 0$. By hypothesis, there exists $u \in \SF{G}{p,q}{\frk h,L}$ satisfying $\dext' u = 0$ and $v - u = \dext' \alpha$. By applying $A$ on both sides, we obtain $Av - Au = \dext'(A\alpha)$. Since $Av = 0$ and $Au = u$, we have $u = -\dext'(A\alpha)$ and so
$v = \dext'(\alpha - A\alpha)$.
\end{proof}

\begin{exa}
 In the complex case, when the group $G$ is semisimple, it is possible to use a result by Bott \cite{bott1957homogeneous} to see that we always have property (K). The application of the Bott's result was made explicit by Pittie on \cite{pittie1988dolbeault}.
\end{exa}

\begin{exa}
 Consider the Example \ref{exa:torus_liouvile} and notice that, if $(\xi_0, \eta_0) \neq (0,0)$, then $u = e^{i (\xi x + \tau t)  } \dext x \in \SF{\T^2}{0,1}{L}$ satisfies $Au = 0$. If  $(\xi_0, \eta_0)$ is a zero of $\xi - \mu\tau = 0$,
 we \emph{cannot} solve the problem $\dext' f = u$ for $f \in \SF{\T^2}{0,1}{L}$. This can be easily seen by using Fourier series.
\end{exa}

\begin{lem}
 If the involutive structure $\frk h$ satisfies property (K) in degree $(p,q)$, then the mapping  $i : H^{p,q}_{L}(G; \frk h) \to H^{p,q}_{}(G; \frk h)$ is surjective.
\end{lem}

Therefore, we have the following theorem:

\begin{thm}
 If the involutive structure $\frk h$ satisfies property (K) in degree $(p,q)$, the cohomologies $H^{p,q}(G; \frk h)$ and $H^{p,q}_{ L}(G; \frk h)$ are isomorphic.
\end{thm}


\subsection{The Hodge decomposition for the operator \texorpdfstring{$\dext'$}{d'}}

Let $W_1, \ldots, W_N$ be a basis for $\frk g$ with dual basis $w_1, \ldots, w_N$. We endow $\frk g$ with the Hermitian inner product making the basis $W_1, \ldots, W_N$ an orthonormal basis, that is
$W_j \cdot W_k \doteq \delta_{jk}$ for $j,k=1,2, \ldots N.$ Notice that $W_j \cdot W_k = w_j(W_k)$ for $j,k = 1,2, \ldots N.$

We denote by $\Lambda^p \frk g$ the $p$th exterior product of $\frk g$ and we recall that, by using the universal property, we can naturally identify $\Lambda^p \frk g^*$ with $(\Lambda^p \frk g)^*$, the dual of $\Lambda^p \frk g$. We write $W_I = W_{i_1} \wedge \ldots \wedge W_{i_p} \in \Lambda^p \frk g$ and denote its dual by $ w_I =  w_{i_1} \wedge \ldots \wedge w_{i_p}$. Now we have that $\{w_I: |I| = p\}$ is a basis for $\Lambda^p \frk g^*$ and we can define an Hermitian inner product on $\Lambda^p \frk g^*$ by $w_I \cdot w_J = w_I(W_J)$, for all $I,J$ with $ |I|=|J| = p.$

If $u \in \SF{G}{p}{}$, we write $u = \sum_{|I| = p} u_I w_I$ and, for each $x \in G$, we have $u_x = \sum_{|I| = p} u_I(x) w_I \in \Lambda^p \frk g^*.$
We endow $\SF{G}{p}{}$ with the inner product 
$(u,v) \doteq \int_G u_x \cdot v_x \dint \mu(x)$
for $u, v \in \SF{G}{p}{}$.

Without loss of generality, we assume that the first $n$ elements of $W_1, \ldots, W_N$ form a basis for $\frk h$. Also, we denote these elements by $L_1, \ldots, L_n$ and the other elements by $M_1, \ldots, M_m$. That is, $L_j = W_j$, for $j=1,\dots,n$, and $M_j = W_{n+j}$, for $j=1, \ldots, m$.

We denote the dual basis for $\{L_1, \ldots, L_n,M_1, \ldots, M_m\}$ by $\{\tau_1, \ldots, \tau_n,\zeta_1, \ldots, \zeta_m\}$. With this basis, each element of $\SF{G}{p,q}{}$ can be identified with an element in $\SF{G}{p+q}{}$. Thus, each $\SF{G}{p,q}{}$ is endowed with an Hermitian inner product induced by the inner product in $\SF{G}{p,q}{}$.

We finally endow $\SF{G}{p,q}{}$ with the inner product
$(u,v) \doteq \int_G u(x) \cdot v(x) \dint \mu(x)$ 
for $u, v \in \SF{G}{p,q}{}$.

With this inner product, for each degree $(p,q)$, we define the formal adjoint of the operator $\dext' : \SF{G}{p,q}{} \to \SF{G}{p,q+1}{}$ as the unique operator $\delta' : \SF{G}{p,q+1}{} \to \SF{G}{p,q}{}$ such that
$ (\dext'u, v) = (u, \delta'v),$
for all $u \in \SF{G}{p,q}{}$ and $v \in \SF{G}{p,q+1}{}$.

For each pair $p,q$ with $p,q \geq 0$, we define the Hodge operator, also known as Box operator, $\Box^{p,q} : \SF{G}{p,q}{} \to \SF{G}{p,q}{}$
by $\Box^{p,q} = \delta' \circ \dext' + \dext' \circ \delta'.$


When $\frk h$ is elliptic, we have the decomposition:
$$ \SF{G}{p,q}{} = \ker \Box^{p,q} \oplus \dext^{p,q-1}(\SF{G}{p,q-1}{}) \oplus \delta^{p,q}(\SF{G}{p,q+1}{}).$$

With this decomposition, we have: suppose that $u \in \SF{G}{p,q}{}$ is such that $\dext' u = 0$. Clearly, we have that $(u,v) = 0$ for all $v \in \delta^{p,q}(\SF{G}{p,q+1}{})$. Thus, if we show that $(u,v) = 0$ for all $v \in \ker \Box^{p,q}$, we have that $u$ is $\dext'$-exact.


\subsubsection{Lie derivatives}

In this section we briefly recall the definition of Lie derivative and explain how it can be extended to act on the spaces $\SF{G}{p,q}{}$.

Let $X$ be a real smooth vector field on $G$. We denote by $\Phi^X$ the flow of $X$ in $G$, that is, $\Phi^X(x, t)$ is a smooth curve on $G$ satisfying $\Phi^X(x, 0) = x$ and $ \dext / \dext t |_{t=0} \Phi^X(x, t) = X(x)$. We also use the notation $\Phi^X_t(x) = \Phi^X(x, t)$. If $f \in \mathcal C^\infty(G)$ is a smooth function, the Lie derivative of $f$ in $x \in G$ is defined by
$$ (\mathscr{L}_X f)(x) = \left. \frac{ \dext }{ \dext t} \right|_{t=0} f(\Phi^X(t,x)).$$
Also, if $u \in \SF{G}{q}{}$ is a smooth $p$-form, its Lie derivative in $x$ is defined by
$$ (\mathscr{L}_X u)_x = \left. \frac{ \dext }{ \dext t} \right|_{t=0} ((\Phi^X_t)^* u)_x.$$

We recall that the Lie derivative has the following properties:
\begin{enumerate}
	\item If $u$ and $v$ are two smooth forms, it holds that
$ \mathscr{L}_X (u \wedge v) = (\mathscr{L}_X u) \wedge v + u \wedge (\mathscr{L}_X v).$
	\item The Lie derivative commutes with the exterior derivative, that is, $ \mathscr{L}_X \circ \dext = \dext \circ \mathscr{L}_X.$
	\item Cartan's Magic Formula:
$ \mathscr{L}_X u = \imath_X (\dext u) + \dext (\imath_X u).$
\end{enumerate}

Notice that, each $Z \in \frk g$ can be written as $Z = X + i Y$ with $X, Y \in \frk g_\R$, and we can define
$$ \mathscr{L}_Z u = \mathscr{L}_X u + i \mathscr{L}_Y u.$$

If $Z \in \frk h$, we can define the Lie derivative on $\SF{G}{p,q}{}$. In fact, if $Z \in \frk h$ and $u \in N^{p,q}(G; \frk h)$, then $\imath_Z u \in N^{p,q-1}(G; \frk h)$. Thus, we have an operator $ \imath_X : N^{p,q}(G; \frk h) \to N^{p,q-1}(G; \frk h)$ satisfying $\imath_Z (N^{p+1,q-1}(G; \frk h)) \subset N^{p+1,q-2}(G; \frk h).$
Therefore, it defines a map
$ \imath_Z : \SF{G}{p,q}{} \to \SF{G}{p,q-1}{}$ such that $\imath_Z \circ \pi^{p,q} = \pi^{p,q-1} \circ \imath_Z$ and so we have $$
\begin{aligned}
  \pi^{p,q} ( \mathscr{L}_Z u ) = & \pi^{p,q}(\imath_Z (\dext u)) + \pi^{p,q}(\dext (\imath_Z u)) \\
	= & \imath_Z (\pi^{p,q+1}(\dext u)) + \dext' ( \pi^{p,q-1}(\imath_Z u)) \\
	= & \imath_Z (\dext' \pi^{p,q}(u)) + \dext' ( \imath_Z \pi^{p,q}(u)).
\end{aligned}
$$

In conclusion, if $Z$ is an element of the involutive structure, we can define the operator $ \mathscr{L}'_Z : \SF{G}{p,q}{} \to \SF{G}{p,q}{}$
by $\mathscr{L}'_Z u = \imath_Z (\dext' u) + \dext' (\imath_Z u). $ Notice that if $u \in \SF{G}{p,q}{}$ is a right-invariant form, then $\mathscr{L}'_Z u = 0$ if $Z$ is a left-invariant vector field.


\subsubsection{Relation between \texorpdfstring{$\Box$}{Box} and \texorpdfstring{$\Delta$}{Laplacian}}

We recall that the Laplace-Beltrami operator is defined by
$ \Delta \doteq \dext \circ \delta + \delta \circ \dext.$

Notice that the box operator $\Box$ associated to the complex \eqref{h_complex} 
 satisfies $\pi^{p,q} \circ \Delta = \Box \circ \pi_{p,q}$ for all elements in $N^{p,q}(G; \frk h)$. Notice that, if $\pi^{p,q}(u) \in \SF{G}{p,q}{}$, then $\Box \circ \pi^{p,q}(u) = \pi^{p,q} (\Delta u)$. Therefore the operator $\Box$ inherits some properties from $\Delta$. The property we need is the following
$$ \mathscr{L}'_Z \Box \pi^{p,q}(u) = \mathscr{L}'_Z \pi^{p,q}(\Delta u) = \pi^{p,q}(\mathscr{L}_Z \Delta u) = \pi^{p,q}(\Delta \mathscr{L}_Z u),$$ which implies the following:
\begin{enumerate}
    \item The operator $\mathscr{L}'_Z$ commutes with $\Box^{p,q}$.
    \item If $\Box u = 0$, then $ \Box \mathscr{L}'_Z u = 0$, $\dext' \mathscr{L}'_Z u = 0$, and $\delta' \mathscr{L}'_Z u = 0.$
    \item If $u \in \ker \Box$, then $\mathscr{L}'_Z u = 0$ for all $Z \in \frk h$.
\end{enumerate}

\begin{prp}
\label{chp:comp_lie:hypo_p_zero} Let $\frk h$ be a hypocomplex structure. If $u \in \mathcal C^{\infty}(G; \Lambda^{p,0})$ is such that  $\mathscr{L}'_Z u = 0$ for all $Z \in \frk h$, then $u$ is left-invariant.
\end{prp}

\begin{proof}
Let $u \in \mathcal C^{\infty}(G; \Lambda^{p,0})$ and assume that $\mathscr{L}'_Z u = 0$ for all $Z \in \frk h$. We write $ u = \sum_{|I|=p} u_I \zeta_I $, in which $u_I \in \mathcal C^{\infty}(G)$. Notice that
 $\mathscr{L}'_Z \left( \sum_{|I|=p} u_I \zeta_I \right) =  \sum_{|I|=p} \left\{ \mathscr{L}'_Z (u_I) \zeta_I + u_I (\mathscr{L}'_Z \zeta_I) \right\}.$

By definition, we have $\mathscr{L}'_Z \zeta_I = \dext' (\imath_Z \zeta_I) + \imath_Z(\dext' \zeta_I)$. Since $\imath_Z \zeta_I$ is zero, we have $\dext' (\imath_Z \zeta_I) = 0$. Also, there exist constants $a_{Ij} \in \C$ such that $\dext' \zeta_I = \sum_{j=1}^n c_{Ij} \zeta_I \wedge \tau_j.$ Thus, by applying $\imath_Z$ on both sides, we obtain
$$\imath_Z(\dext' \zeta_I) = \imath_Z(\sum_{j=1}^n c_{Ij} \zeta_I \wedge \tau_j) = \sum_{j=1}^n c_{Ij} \imath_Z(\zeta_I \wedge \tau_j) = 0$$
because $\imath_{Z}(\zeta_I \wedge \tau_j) = 0$ when computed in any element of $\frk h$.

Therefore, $\mathscr{L}'_Z (\sum_{|I|=p} u_I \zeta_I) =  \sum_{|I|=p} \mathscr{L}'_Z (u_I) \zeta_I = 0$ and so $\mathscr{L}'_Z u_I = 0$ for all $Z$. Since $\frk h$ is hypocomplex, we conclude that $u_I$ is a constant. Therefore, the form $u$ is left-invariant.
\end{proof}

Notice that the Lie derivative can be extended to work with currents and the same proof yields the following result:

\begin{cor}
\label{cor:hypo_p_zero_in_Eprime_two} Let $\frk h$ be a hypocomplex structure. If $u \in \mathcal D'(G; \Lambda^{p,0})$ is such that  $\mathscr{L}'_Z u = 0$ for all $Z \in \frk h$, then $u$ is left-invariant.
\end{cor}

Notice that the results above implies Theorem \ref{int:comp_lie:hypo_subelliptic}.


\begin{prp}
\label{prp:comp_lie:hypo_subelliptic_torus}
 Let $\frk h$ be a left-invariant elliptic structure on a torus $\mathbb{T}$. If $u \in \SF{\T}{p,q}{}$ is such that $\mathscr{L}'_Z u = 0$ for all $Z \in \frk h$, then $u$ is left-invariant.
\end{prp}

\begin{proof}
 If $u \in \SF{\T}{p,q}{}$ is such that $\mathscr L'_Z u = 0$ for all $Z \in \frk h$, by writing $ u = \sum_{|I|=p} \sum_{|J|=q} u_{IJ} \zeta_I \wedge \tau_J,$
 we have that $$ \mathscr L'_Z u = \sum_{|I|=p} \sum_{|J|=q} \mathscr L'_Z ( u_{IJ} ) \zeta_I \wedge \tau_J = 0. $$
 Thus, $\mathscr L'_Z (u_{IJ}) = 0$. Since $\frk h$ is elliptic, we have that $u_{IJ} \in \C$ and the theorem is proved.
\end{proof}

Now we prove Theorem \ref{int:comp_lie:hypo_subelliptic_torus_2}.

\begin{proof}[Proof of \ref{int:comp_lie:hypo_subelliptic_torus_2}]
 Let $[u] \in H^{p,q}_{\mathcal C^\infty}(\T; \frk h)$. Since $\Box$ elliptic, we can assume $\Box u = 0$ and $\mathscr L'_Z u = 0$ for all $Z \in \frk h$, since $\frk h$ is actually bi-invariant, it follows that $u$ is left-invariant. The proof is concluded.
\end{proof}


\subsection{Application of Serre duality}

For this application, we need to extend the averaging operator acting on forms to an operator acting on currents, which can be done by transposition. Since the averaging operator is formally self-adjoint, as we will prove in Lemma \ref{lem:formally_self-adjoint}, we can easily see that this extension is well-defined.

Te operator $A$ is formally self-adjoint. We prove this fact in the following lemma:

\begin{lem}
\label{lem:formally_self-adjoint}
 The operator $A : \SF{G}{p}{} \to \SF{G}{p}{}$ defined in Lemma \ref{averaging_operator_lemma} is formally self-adjoint.
\end{lem}

\begin{proof}
Let $u, v \in \SF{G}{p}{}$ and write
$ u = \sum_{|I|=p} u_I w_I,~ v = \sum_{|J|=p} v_J w_J.$

Notice that, for $X_1, \ldots, X_p \in \frk g$, with a simple computation we obtain
$$
\begin{aligned}
	(Av)_g(X_1, \ldots, X_p)
		& = \sum_{|J|=p} \left(\int_G v_J \dint \mu(x) \right) w_J(X_1, \ldots, X_p).  \\
\end{aligned}
$$

Thus, it holds that
$$\begin{aligned}
	(u,Av)	& = \sum_{|I|=p}\sum_{|J|=p} \int_G u_I(g) \left(\int_G \overline{ v_J \circ L_x(g)} \dint \mu(x) \right) w_I \cdot w_J \dint \mu(g) \\
		& = \sum_{|I|=p}\sum_{|J|=p} \left( \int_G \int_G u_I(g) \overline{ v_J \circ L_x(g)} \dint \mu(x)  \dint \mu(g) \right) w_I \cdot w_J \\
		& = \sum_{|I|=p}\sum_{|J|=p} \left( \int_G \int_G u_I \circ L_x(g) \overline{ v_J (x)} \dint \mu(x)  \dint \mu(g) \right) w_I \cdot w_J \\
		& = (Au,v).
\end{aligned}$$
\end{proof}

The following lemma is going to be useful:

\begin{lem}
 Let $u \in \mathcal D'(G; \Lambda^k)$ and write $\Sigma_{|I| = k} u_I w_I$. Then $Au = \Sigma_{|I| = k} (Au_I) w_I.$
\end{lem}

\begin{proof}
 Let $\phi \in \mathcal C^{\infty}(G; \Lambda^{N-k})$ and write $\phi = \Sigma_{|I'| = N-k} \phi_{I'} w_{I'}$. By definition, we have
$$
\begin{aligned}
	(Au)(\phi) & = u(A \phi) \\
            & = (\Sigma_{|I| = k} u_I w_I)(\Sigma_{|I'| = N-k} (A\phi_{I'}) w_{I'}) \\
            & = \Sigma_{|I| = k} \Sigma_{|I'| = N-k} (Au_I)(\phi_{I'}) w_I \wedge w_{I'} \\
            & = (\Sigma_{|I| = k} (Au_I) w_I)(\phi).
\end{aligned}
$$
Thus, the proof is completed.
\end{proof}



\begin{lem}
 Let $u \in \mathcal D'(G; \Lambda^k)$ and suppose that $Au - u = 0$. Then $u$ is a real-analytic left-invariant form.
\end{lem}

\begin{proof}
 Let $u \in \mathcal D'(G; \Lambda^k)$ and write $u = \Sigma_{|I| = k} u_I w_I$. We note that, if $0 = Au - u$, then $Au_I - u_I = 0$, which means that each $u_I$ is a constant.
\end{proof}

Now we have everything we need to prove Theorem \ref{int:li:thm:duality_of_left}.


\begin{proof}[Proof of Theorem \ref{int:li:thm:duality_of_left}]
Since the operators \eqref{eq:two_operators_closed_range} have closed range, by Serre duality, we have $ H^{p,q}_{\mathcal C^{\infty}}(G; \frk h)^* \cong H^{m-p,n-q}_{\mathcal D'}(G; \frk h).$
Let $[u] \in H^{m-p,n-q}_{\mathcal D'}(G; \frk h)$. For every $[v] \in H^{p,q}_{\mathcal C^{\infty}}(G; \frk h)$, we have
$$ (Au - u)(v) = Au(v) - u(v) = u(Av) - u(v) = u(Av - v).$$

Suppose that every class in $H^{p,q}_{\mathcal C^{\infty}}(G; \frk h)$ has a representative that is left-invariant. Thus, we can assume that $v$ is left-invariant, that is, $Av - v = 0$. Then, $[Au - u] = 0$ in $H^{m-p,n-q}_{\mathcal D'}(G; \frk h)$.
The other direction of the equivalence follows by applying the exact same argument.
\end{proof}


\begin{rmk}
\label{chp:li:rmk:const} Notice that, since $[Av] = [v]$ in $H^{m-p,n-q}_{\mathcal D'}(G; \frk h)$ and $Av$ is left-invariant, we actually have $$H^{m-p,n-q}_{\mathcal D'}(G; \frk h) \cong  H^{m-p,n-q}_{\mathcal C^{\infty}}(G; \frk h) \cong H^{m-p,n-q}_{L}(G; \frk h).$$
\end{rmk}



Now we prove Theorem \ref{int:left_invariance_on_degree}.


\begin{proof}[Proof of Theorem \ref{int:left_invariance_on_degree}]
If $[u] \in H^{p,0}_{\mathcal{D}'}(G; \frk h)$, then clearly $\Box u = 0$, which means that $\mathscr L'_Z u = 0$ for all $Z \in \frk h$. Therefore, $u$ is left-invariant by Corollary \ref{cor:hypo_p_zero_in_Eprime_two}. By combining this with Theorem \ref{int:li:thm:duality_of_left} and Remark \ref{chp:li:rmk:const}, we have left-invariance of cohomologies in top degree, that is,
$H^{p,n}_{\mathcal C^{\infty}}(G;\frk h) = H^{m-p,n}_{L}(G; \frk h).$\end{proof}

\begin{rmk}
\label{rmk:adele_labels_everything}
 The operator $\dext' : \mathcal C^{\infty}(G; \Lambda^{p,q}) \xrightarrow{} \mathcal C^{\infty}(G; \Lambda^{p,q+1})$ is known to have closed range in many situations. For example, it always has closed range in the case of elliptic structures. For hypocomplex structures, the situation is more delicate, but in \cite{cordaro2019top} we found some sufficient conditions to ensure that $\dext' : \mathcal C^{\infty}(G; \Lambda^{p,q}) \xrightarrow{} \mathcal C^{\infty}(G; \Lambda^{p,q+1})$ has closed range. Notice that, due to the left-invariance of the structures, we only need to verify the hypothesis in the origin. Also in \cite{cordaro2019top}, it is shown that the operator $\dext' : \mathcal B(G; \Lambda^{p,n-1}) \xrightarrow{} \mathcal B(G; \Lambda^{p,n})$
 always has closed range.
\end{rmk}

For hypocomplex structures, the operator $\dext' : \mathcal B(G; \Lambda^{p,n-1}) \xrightarrow{} \mathcal B(G; \Lambda^{p,n})$ always has sequentially closed range and so the operator $\dext' : \mathcal C^\omega(G; \Lambda^{p,0}) \xrightarrow{} \mathcal C^\omega(G; \Lambda^{p,1})$ has closed range. Therefore, by a direct adaptation of Theorem \ref{int:comp_lie:hypo_subelliptic}, it holds that $H^{p,0}_{\mathcal C^{\omega}}(G;\frk h) = H^{p,0}_{L}(G; \frk h)$ and this implies that $H^{p,n}_{\mathcal B}(G;\frk h) = H^{p,n}_{L}(G; \frk h)$. This result can be extended to other classes of smooth functions. We briefly introduce them in the following.

Let $\mathcal M$ denote a Gevrey or Denjoy-Carleman class. We refer to \cite{fujita1995gevrey} and \cite{dasgupta2016eigenfunction} for an introduction to this setting in compact manifolds and on compact Lie groups. We denote by $\mathcal C^{\mathcal M}(G)$ the space of smooth function belonging to the class $\mathcal M$. This space can be endowed with a topology of ({\bf{DFS}}) spaces. The topological dual of $\mathcal  C^{\mathcal M}(G)$ is denoted by $\mathcal D'_{\mathcal M}(G)$ and, obviously, has a topology of a ({\bf{FS}}) space. When $\mathcal M$ is the class of real-analytic functions, we denote it by $\mathcal C^\omega(G)$ with its dual denoted by $\mathcal B(G)$.

Since $\mathcal D'_{\mathcal M}(G; \Lambda^{p,q}) \subset \mathcal B(G; \Lambda^{p,q})$, we have a natural homomorphism $ H^{p,n}_{\mathcal D'_{\mathcal M}}(G;\frk h)  \to H^{p,n}_{\mathcal B}(G;\frk h) $. This homomorphism is injective. In fact, let $[u] \in H^{p,n}_{\mathcal D'_{\mathcal M}}(G;\frk h)$ and suppose that $[u] = 0$ in $H^{p,n}_{\mathcal B}(G;\frk h)$. Since $H^{p,n}_{\mathcal B}(G;\frk h) = H^{p,n}_{L}(G; \frk h)$, there exists a left-invariant $v \in \mathcal B(G; \Lambda^{p,n})$ such that $u - v = \dext' w$ for some $w \in \mathcal B(G; \Lambda^{p,n})$ and since $v$ is left-invariant, we obtain $v \in \mathcal D'_{\mathcal M}(G; \Lambda^{p,q})$.

Notice that $A(u - v) = Au - v = \dext' (Aw)$ with $Aw$ being a section with real-analytic coefficients. Therefore $Aw \in \mathcal D'_{\mathcal M}(G; \Lambda^{p,q-1})$ and so $[u] = [v] = [Au]$ in $H^{p,n}_{\mathcal D'_{\mathcal M}}(G;\frk h)$. Now, $[v] = 0$ means that there exists $w \in \mathcal B(G; \Lambda^{p,n-1})$ such that $\dext' w = v$. Since $v$ is left-invariant, we obtain $v = Av = A (\dext' w) = \dext'(Aw)$ with $Aw \in \mathcal D'_{\mathcal M}(G; \Lambda^{p,n-1})$. Therefore $[u] = 0$ in $H^{p,n}_{\mathcal D'_{\mathcal M}}(G;\frk h)$.

Now, $\dim H^{p,n}_{\mathcal D'_{\mathcal M}}(G;\frk h) < \infty$ and so the codimension of
 $\dext' : \mathcal D'_{\mathcal M}(G; \Lambda^{p,n-1}) \xrightarrow{} \mathcal D'_{\mathcal M}(G; \Lambda^{p,n})$ is finite and we conclude that this map has closed range. We proved Theorem \ref{int:awesome}:
 
\section{Involutive structures on homogeneous manifolds}
\label{sec:homogeneous}
In this section, we introduce the concept of invariant involutive structures on homogeneous manifolds and show how to construct some examples. The main reason for introducing this concept is to simplify the study left-invariant involutive structures on Lie groups.

The main idea we want to explore is the following: let $G$ be a compact Lie group endowed with an elliptic involutive structure $\frk h$. Also, let $\frk k_\R = \frk h \cap \frk g_\R$ and consider the group $K = \exp(\frk k_\R)$. We assume that $K$ is closed, so we can define a homogeneous manifold $\Omega = G / K$. This homogeneous manifolds inherits a complex involutive structure from $G$ via the quotient map.

With all these ingredients in hand, we want to see how much we can infer about the cohomology spaces $H^{p,q}(G;\frk h)$ from the de Rham cohomology of $K$ and from the Dolbeault cohomology of $\Omega$. This is done via a spectral sequence. In order to use this spectral sequence, we need some further assumptions. One of the assumptions is that the orbits of the elliptic structure $\frk h$ is closed and the other one is that $\frk h$ can be decomposed into a direct sum of two Lie subalgebras: one is the algebra $\frk k = \frk h \cap \frk g$ and the other is an ideal in $\frk h$. We show that this tecnique can always be applied in a situation similar to a famous conjecture by Treves.

To conclude, we introduce the concept of Lie algebra cohomology and use it to state a theorem by Bott, which, under certain assumptions, gives us algebraic information about the Dolbeault cohomology of $\Omega$. Under certain algebraic and topological restrictions on $G$, Bott's theorem, along with the spectral sequences, give us a complete algebraic description of $H^{p,q}(G;\frk h)$.


\subsection{Homogeneous manifolds and invariant involutive structures}

Let $\Omega$ be a homogeneous manifold for a Lie group $G$. We recall that $\Omega$ is a smooth manifold and that there exists a smooth map $T : G \times \Omega \to \Omega,$ called \emph{left-action} of $G$ on $\Omega$.
Other notations for $T$ are $T(g, x) = T_g x = g \cdot x$.
The map $T$ satisfies the following properties:
$g \cdot (g' \cdot x) = (gg') \cdot x \text{ and } e \cdot x = x,$ in which $e \in G$ is the identity of $G$ and $T$ is assumed to be \emph{transitive}, that is, for any two points $x,y \in \Omega$, there is an element $g$ such that $g \cdot x = y$.

\begin{exa} Any Lie group acts transitively on itself by left multiplication.
\end{exa}

\begin{exa}
If $G$ is any Lie group and $H$ is a closed Lie subgroup, the space $\Omega = G / H$ is a homogeneous manifold with $G$ acting on $\Omega$ by
$g \cdot (g' H) = (g g') \cdot H.$

By Theorem 7.19 on \cite{lee2003smooth}, we know that all examples of homogeneous manifolds are equivalent to this example and, thus, we have that $\Omega$ is an analytic manifold and the map $T$ is an analytic map.
\end{exa}

From now on, $\Omega$ is always assumed to be a connected $G$-homogeneous manifold.

We want to use involutive structures on homogeneous manifolds. In order to take into consideration the symmetries given by the group action, we need to consider involutive structures that carry information about the action. We then introduce the following concept:

\begin{dfn}
 An involutive structure $\mathcal V \subset \C T\Omega$ is said to be \emph{invariant by the action of $G$} if $(T_g)_*X \in \mathcal V_{g \cdot x}$ for all $X \in  \mathcal V_x.$
\end{dfn}

The most obvious example of an invariant structure is $\mathcal{V} = \C T \Omega$. Next, we construct some examples of invariant involutive structures.


Let $G$ be a Lie group and denote its Lie algebra by $\frk g_{\R}$. We denote by $\frk g$ the complexification of $\frk g_{\R}$.


Let $X$ be a vector field on a smooth manifold $\Omega$.
Suppose that $\mathcal V \subset \C T\Omega$ is a vector bundle. We say that $\mathcal V$ is preserved by the vector field $X$ if, for all smooth sections $Y$ of $\mathcal V$, we have that $[X,Y]$ is a smooth section of $\mathcal V$. 
If $F$ is a set of vector fields, we say that $\mathcal V$ is preserved by $F$ if $\VV$ is preserved by all $X \in F$.

%
%

\begin{prp}
\label{prp:inv_defined_on_quotient}
Let $G$ be a Lie group with $K \subset G$ being a connected closed subgroup. Consider $\pi : G \to G/K = \Omega$ the quotient map. Suppose that $\frk h \subset \frk g$ is a complex Lie subalgebra and that $\frk h$ is preserved by $\ker \pi_*$. Then, $\pi_* \frk h$ defines an invariant involutive structure on $\Omega$.
\end{prp}

\begin{proof}
 It follows from Problem 2.57 of \cite{gadea2012analysis} that $\pi_* \frk h$ is an involutive vector bundle. The invariance follows directly from the definition.
\end{proof}

The following definition is going to be useful.

\begin{dfn}
    Let $G$ be a Lie group and let $\frk k \subset \frk g_\R$ be a subalgebra. We say that $\frk k$ is a \emph{closed subalgebra} of $\frk g_\R$ if the group $K = \exp_G(\frk k)$ is a closed subgroup of $G$. When $\frk h$ is a left-invariant involutive structure, we say that $\frk h$ has \emph{closed orbits} if $\frk k_\R = \frk h \cap \frk g_\R$ is a closed subalgebra.
\end{dfn}

The following lemma is one of the key technical results in this section.

\begin{lem}
    Let $G$ be a Lie group endowed with a left-invariant elliptic structure $\frk h$ having closed orbits. Then, the homogeneous space $\Omega = G/K$, with $K = \exp(\frk h \cap \frk g_\R)$, has a natural complex structure given by $\pi_*(\frk h)$.
\end{lem}

\begin{proof}
    Since $\pi$ is a real map, we have
    $\pi_*(\frk h) + \overline{ \pi_*(\frk h) } = \pi_*(\frk h) + \pi_*(\overline{ \frk h } ) =  \pi_*(\frk h + \overline{ \frk h } ) = \pi_*(\frk g) = \C T \Omega.$
    The last equality holds because $\pi_*$ is surjective, implying that $\pi_*(\frk h)$ is elliptic. Notice that, by dimensional reasons, the involutive structure $\pi_*(\frk h)$ is a complex structure.
\end{proof}

We need the following technical lemma:

\begin{lem}
\label{lem:elliptic:important}
Let $G$, $\frk h$, $\frk k$, $K$, and $\frk u$ be as in Theorem \ref{int:thm:homogeneous:converges_in_E_2}. Let $u$ be a left-invariant closed $s$-form in $K$. Then $u$ can be extended to a $\dext'$-closed form in $G$. This extension, when restricted to a leaf $gK$, is a left-invariant form on that leaf $gK$.
\end{lem}

\begin{proof}
\label{lem:cohomology_extension}
Since $u$ is left-invariant, we can regard it as an element of the dual of $\Lambda^s\frk k$ and we can extend it to the dual of $\Lambda^s\frk g$ by defining it as zero if any of its arguments it in $\frk u$ or $\overline{\frk u}$. Let $w_g = (L_{g^{-1}})^* u$. Notice that $$(L_g)^* : H^s(gK) \to H^s(K)$$ is an isomorphism. Thus, $w$ is a smooth $s$-form in $G$ and the restriction of $w$ to a leaf $gK$ is representive of a cohomology class in that leaf $gK$.

We are going to see that $w$ is also $\dext'$-closed. We just need to show that $w$ satisfies $\dext w \in N^{1,r}(\frk g; \frk h)$. Let $X_1, \ldots, X_{r+1} \in \frk h$. Since the exterior derivative commutes with the pullback, we have
$$ \dext w_g = \dext (L_{g^{-1}})^* u = (L_{g^{-1}})^* u (\dext u).$$
Now, since every term $X_j u (X_1, \ldots, \hat {X_j}, \ldots, X_{r+1})$ is zero, we only have to do the following computation
$$
\begin{aligned}
\dext u  (X_1,  \ldots, X_{r+1})  & = \sum_{j=1}^{r+1} (-1)^{j+1} X_j u (X_1, \ldots, \hat {X_j}, \ldots, X_{r+1}) \\
    & + \sum_{j < k} (-1)^{j+k+1} u ([X_j, X_k], X_1, \ldots, \hat {X_j}, \ldots, \hat {X_k}, \ldots, X_{r+1}) \\
    & = \sum_{j < k} (-1)^{j+k+1} u([X_j, X_k], X_1, \ldots, \hat {X_j}, \ldots, \hat {X_k}, \ldots, X_{r+1}).
\end{aligned}
$$

If we assume that $X_j \in \frk k$ for all $j$, then $(\dext  w_g)  (X_1,  \ldots, X_{r+1}) = 0$ because $\dext u = 0$.

However, if we assume that $X_1 \in \frk u$ and that $X_j \in \frk k$ for j > 1, then
$$
\begin{aligned}
\sum_{j < k} (-1)^{j+k+1} & u([X_j, X_k], X_1, \ldots, \hat {X_j}, \ldots, \hat {X_k}, \ldots, X_{r+1}) \\ &= \sum_{1 < k} (-1)^{1+k+1} u([X_1, X_k], X_2, \ldots, \hat {X_j}, \ldots, \hat {X_k}, \ldots, X_{r+1}) = 0
\end{aligned}
$$
because $\frk u$ is an ideal and so $[X_1, X_k] \in \frk u$.

Finally, if we assume that two or more elements are $X_1,X_2 \in \frk k^\perp$, we can use the fact that $\frk u$ is an ideal and an argument similar to the one above to prove that we have $(\dext  w_g)  (X_1(g),  \ldots, X_{r+1}(g)) = 0.$ Therefore, $\dext  w_g \in N^{1,q}(G; \frk h)$.
\end{proof}

Now we have everything we need to prove Theorem \ref{int:thm:homogeneous:converges_in_E_2}.

\begin{proof}[Proof of Theorem \ref{int:thm:homogeneous:converges_in_E_2}]



    Let $[\beta] \in H^s(K; \C)$. We can always choose $\beta$ left-invariant, so that, by the last Lemma \ref{lem:cohomology_extension}, there exists an extension of $\beta$ to $G$ such that $\dext' \beta' = 0$. This is a cohomology extension and, according to Leray-Hirsh Theorem \citep[Theorem 9 of Section 7, Chapter 5]{spanier1994algebraic}, the proof is completed.
\end{proof}

Notice that, in Equation \eqref{int:eq:decomposition_in_E_2}, the terms $H^s(K; \C)$ can always be computed by using algebraic methods. On the other hand, the terms $H^{p,r}(\Omega; \mathcal V)$ can be more complicated to compute.

Our next objective is to find some conditions for which the terms $H^{p,r}(\Omega; \mathcal V)$ can be computed using only algebraic methods. The first thing we do is find conditions so that $\Omega$ is a compact connected Riemann surface. That is exactly what we do in Theorem \ref{int:omegaIsRieaman}.

\begin{proof}[Proof of Theorem \ref{int:omegaIsRieaman}]
    Let $\ip{,}$ be a Hermitian extension of any ad-invariant inner product on $\frk g_\R$.
		We write $\frk k = \frk h \cap \overline{\frk h}$ and we define $\frk k^\perp = \{Z \in \frk h : \ip{Z,W} = 0 ~\forall W \in \frk h\}.$ Since $\dim \frk h = \dim \frk g - 1$, we have that $\dim \overline{\frk k^\perp} = 1$ and we clearly have $\frk g = \frk k \oplus \frk k^\perp \oplus \overline{\frk k^\perp}.$
Since $\frk k^\perp$ has dimension 1, it is Abelian. More than that, it is an ideal. In fact, we have that, if $T \in \frk k_\R$ and $Z \in \frk h^\perp$, then
		$\ip{[Z,T],U} = \ip{Z,[T,U]} = 0, ~ \forall U \in \frk k,$
		which means that $[Z,T] \in \frk k^\perp$, for all $T \in \frk k_\R$ and, by linearity, this is also true for all $T \in \frk k$. Therefore, if we assume that $K = \exp_G(\frk k_\R)$ is closed, we have all the necessary conditions to apply Theorem \ref{int:thm:homogeneous:converges_in_E_2}. Notice that, in this case, the homogeneous space $\Omega = G / K$ is a compact Riemann surface, which concludes the proof.
\end{proof}

Notice that, in Example \ref{exa:there_are_examples}, we constructed an example satisfying the hypothesis of the Theorem \ref{int:omegaIsRieaman}.

Now we discuss some techniques to deal with the case in which the complex dimension of $\Omega$ is bigger than 1. For this, we introduce the concepts of Lie algebra cohomology and of Lie algebra cohomology relative to a subalgebra. With this new concepts, we can state Bott's theorem, which gives us an algebraic way to compute the Dolbeault cohomology of certain compact homogeneous complex manifolds.


\subsection{Cohomology of Lie algebras}
\label{sec:cohomology_of_lie_algebras}

In this section, we briefly recall some definitions regarding Lie algebra cohomology. These definitions are necessary in order to introduce a theorem by Bott. We refer to \cite{bott1957homogeneous} and \cite{hochschild1953cohomology} for details.

\begin{dfn} Let $\frk g$ be a Lie algebra over $\C$ and $M$ a vector space over $\C$. A \emph{representation} of $\frk g$ is a Lie algebra homomorphism of $\frk g$ into the Lie algebra of all linear transformations of $M$ into $M$.
\end{dfn}

\begin{dfn}
Let $\frk g$ be a Lie algebra over $\C$ and $M$ a vector space over $\C$. We say that the pair $(M, \phi)$ is a \emph{$\frk g$-module} if $\phi$ is a representation of the Lie algebra $\frk g$ into $M$.
We denote by $X \cdot x$ the action of $\phi(X)$ on the element $x \in M$.
\end{dfn}

We denote by $C^p(\frk g; M)$ the set of all alternating multi-linear $p$-forms on $\frk g$ with values in $M$. Also, we denote by $C(\frk g; M)$ the formal sum $\sum_p C^p(\frk g; M)$. We identify $C^0(\frk g; M)$ with $M$. 

For $p > 0$, $u \in C^p(\frk g; M)$, and $X_1, \ldots, X_{p+1} \in \frk g$, we define a $\frk g$-homomorphism $\dext : C^p(\frk g; M) \to C^{p+1}(\frk g; M)$ by the formula
\begin{equation}
\label{eq:algebraic_diff_op}
\begin{aligned}
	\dext u(X_1, \ldots, X_{p+1})
		= & \sum_{j=1}^{p+1} (-1)^{j+1} X_j \cdot u(X_1, \ldots, \hat{X}_j, \ldots, X_{p+1})  \\
		 + & \sum_{j < k} (-1)^{j+k+1} u([X_j, X_k], X_1, \ldots, \hat X_j, \ldots, \hat X_k, \ldots, X_p).
\end{aligned}
\end{equation}
The operator $\dext$ is just the algebraic version of the usual exterior differentiation on smooth manifolds. As expected, we have: if $u \in C^0(\frk g; M) = M$, then $(\dext u)(X) = X \cdot u$ and, obviously, it holds that $\dext^2 = 0$.

Therefore, we have a complex with respect to $\dext$. The cohomology space associated to $\dext$, denoted by $H^*(\frk g; M)$, is called cohomology module of $\frk g$ with coefficients in $M$. It is easy to see that $H^*(\frk g; M)$ is a vector space over $\C$.



If $\frk u \subset \frk g$ is a subalgebra, we denote by $C^k(\frk g, \frk u; M)$ the set of elements of $C^k(\frk g; M)$ that vanishes when have any parameter is in $\frk u$. This subset is stable under the operator $\dext$ and thus we also can introduce its cohomology space, which is denoted by $H^k(\frk g, \frk u; M)$.

\subsection{Bott's Theorem}

    In this section, we state a theorem by Bott \cite{bott1957homogeneous} that is useful when dealing with elliptic involutive structures on compact Lie groups.
   
    Let $U$ and $G$ be complex Lie groups with $U$ closed in $G$. Let $\Omega = G/U$ and suppose that $G$ is connected and $\Omega$ is compact and simply connected.
    
    By a theorem of Montgomery \citep{montgomery1950simply}, if $K$ is a maximal compact subgroup of $G$, under the above conditions, $K$ acts transitively on $\Omega$ and therefore $\Omega$ has another description, namely, $ \Omega = K/H$
    with $H = U \cap K$.
    
    We denote, respectively, by $\frk g, \frk u, \frk k, \frk h$ the complexified Lie algebras of $G$, $U$, $K$ and $H$. Since $G$ is a complex Lie group, we can decompose $\frk g$ into two ideals, that is,  $\frk g = \frk g_\alpha \oplus \frk g_\beta$, in which $\frk g_\alpha$ is the set of all left-invariant vector fields annihilated by all anti-holomorphic differential forms on $G$ and $\frk g_\beta$ is the set of all left-invariant vector fields annihilated by all holomorphic differential forms on $G$.
    
    Let $\alpha : \frk g \to \frk g_\alpha$ be the projection and denote by $\imath : \frk k \to \frk g$ the inclusion of $\frk k$ into $\frk g$. Write $\frk u_* = (\alpha \imath)^{-1}(\frk u)$. Clearly, $\frk u_* \subset \frk k$. In the case where $G$ is the complexification of $K$,
    which always can be assumed by Montgomery's theorem and is the case we are working on, $u_*$ can be identified with $\frk u_\R$, the real Lie algebra of $U$.

    \begin{thm}[Bott's theorem]
    \label{thm:bott}
        Let $G$, $U$, $K$, $H$, and $\Omega$ with the conditions we just established and let $\mathscr O^p$ be the sheaf of local holomorphic $p$-forms on $\Omega$. Then,
        $$ H^q(\Omega; \mathscr O^p) = H^q(\frk u_*, \frk h, \Lambda^p(\frk k / \frk u_*)^*)$$
        with $u_*$ acting on $\Lambda^p(\frk k / \frk u_*)^*$ via adjoint action.
    \end{thm}

\subsubsection{Applications of Bott's theorem}

Now we prove some propositions that are useful when applying Bott's theorem.

The next proposition shows that any compact homogeneous space $\Omega$ endowed with an invariant complex structure can be represented as a quotient of two complex Lie groups.

\begin{prp}
\label{prp:ext_act_complexified} Let $G$ be a connected compact Lie group acting transitively on a smooth manifold $\Omega$. If $\Omega$ is endowed with a complex structure invariant by the action of $G$, then this action extends to a transitive holomorphic action of $G_\C.$
\end{prp}

\begin{proof} We are assuming that the action of $G$ preserves the complex structure, so each automorphism $T_g : \Omega \to \Omega$ is holomorphic and we have a group homomorphism
\begin{equation}
\label{eq:cont_homo_of_groups}
	g \in G \mapsto T_g \in \Aut_\Ho(\Omega).
\end{equation}

Since $\Omega$ is compact, by \cite[Theorem 1.1, Chapter III]{kobayashi2012transformation}, the set $\Aut_\Ho(\Omega)$ is a complex Lie group. The topology on $\Aut_\Ho(\Omega)$ is the compact open topology and, in this case, it is the topology of uniform convergence over compact sets. Since $T : G \times \Omega \to \Omega$ is smooth, particularly it is uniformly continuous. Therefore, the map \eqref{eq:cont_homo_of_groups} is continuous and is automatically a Lie group homomorphism.

Let $(G_\C, \eta)$ be the universal complexification of $G$. By the universal property, there exists a complex Lie group homomorphism $g \in G_\C \mapsto T'_g \in \Aut_\Ho(\Omega)$ such that $T'_{\eta(g)} = T_g$, for all $g \in G$. Since $G$ is compact, we have that $\eta$ is injective. Therefore, $G$ can be identified with $\eta(G)$ and the action $T'$ can be considered an extension of $T$.
\end{proof}

The next proposition shows that, in order to study the cohomology of a left-invariant elliptic involutive structure defined over a compact semisimple Lie group, it is enough to study the cohomology of the structure lifted to the universal covering. The advantage of such approach is to be able to assume that the group in question is simply-connected, thus removing some topological barries.

\begin{prp}
\label{ss:prp:sctoss} Let $G$ be a semisimple compact Lie group and suppose that it is endowed with an elliptic involutive structure $\frk h \subset \frk g$. Then, its universal covering group $G'$ is also compact and admits an elliptic involutive structure $\frk h'$ such that $H^{p,q}(G; \frk h) \cong H^{p,q}(G';\frk h')$.
\end{prp}

\begin{proof}
By Weyl's theorem \citep[Theorem 4.26]{knapp2016representation}), the universal covering $G'$ is compact. We know that the covering map $\pi : G' \to G$ gives an isomorphism between $\frk g$ and $\frk g'$. Via this isomorphism, we define $\frk h' \subset \frk g'$. We also have that $\pi^{-1}(x)$ is finite for every $x \in G$. Therefore, by \cite[Theorem 11.1]{bredon2012sheaf}, we have the required isomorphism and the proof is completed.

In order to give a direct proof, we can construct the required isomorphism as follows. We denote by $\mathcal S_{\frk h}$ the sheaf of local solutions of $\frk h$ and by $\mathcal S_{ \frk h'}$ the sheaf of local solutions of $\frk h'$. Let $\mathcal U'$ be a finite covering of $G'$ consisting of sets satisfying the following conditions: for every $U' \in G'$ and $U = \pi(U')$, it holds that $\pi|_{U'}: U' \to U$ is a diffeomorphism and also, for $q > 0$, it holds that $H^q(U'; \mathcal S_{\frk h'}) = 0$.
The condition about the diffeomorphism is possible because $\pi$ is a covering map and the condition about the sheaf cohomology is possible because $\frk h'$ is elliptic. By taking $\mathcal U = \{ U = \pi(U'): U' \in U\}$, we have a finite open covering for $G$ such that $H^q(U; \mathcal S_{\frk h}) = 0$.

Next we construct a cochain isomorphism between $(C^q(\mathcal U', \mathcal S_{\frk h'}), \delta'^q)$ and $(C^q(\mathcal U, \mathcal S_{\frk h}), \delta^q)$. Since $\mathcal U'$ is finite, we can enumerate its elements $U'_1, U'_2, \ldots$ and we have a one-to-one correspondence with elements of $\mathcal U$, namely $U_1, U_2, \ldots$. Thus, for any simplex of $\sigma \in N(\mathcal U)$, there is an unique associated simplex $\sigma' \in N(\mathcal U')$, that is, if
$\sigma = (i_0, \ldots, i_q)$ and $|\sigma| = U_{i_0} \cap \cdots \cap U_{i_q} $, then $\sigma' = (i_0, \ldots, i_q)$ and $|\sigma'| = U'_{i_0} \cap \cdots \cap U'_{i_q} $.

Let $f' \in (C^q(\mathcal U', \mathcal S_{\frk h'}), \delta'^q)$ and let $\sigma'$ be a $q$-simplex. We define $f_\sigma$ in $|\sigma|$ by $f_\sigma(x) = f'(\pi^{-1}(x))$. The function $f_\sigma$ is well defined because $\pi|_{\sigma'}$ is a diffeomorphism and is in $\mathcal S_{\frk h}$.
We define a map $\phi : C^q(\mathcal U', \mathcal S_{\frk h'}) \to C^q(\mathcal U', \mathcal S_{\frk h'})$ which is obviously bijective and also commutes with the restrictions. Thus, we have that $\delta^q \phi(f') = \phi( \delta'^q f')$ and the map $\phi$ induces an isomorphism between $H^q(\mathcal U',\mathcal S_{\frk h})$ and $H^q(\mathcal U;\mathcal S_{\frk h})$.

Now, by using Leray's Theorem \cite[Section D, Theorem 4]{gunning1965analytic}, for every $q \geq 0$, we have that $H^q(G',\mathcal S_{\frk h'}) \cong H^q(\mathcal U';\mathcal S_{\frk h'})$ and $H^q(G,\mathcal S_{\frk h}) \cong H^q(\mathcal U;\mathcal S_{\frk h})$. Finally, we have $H^q(G',\mathcal S_{\frk h'}) \cong H^q(G,\mathcal S_{\frk h})$. 
\end{proof}

    Now, we assume $G$ to be a compact semisimple Lie group. Combining Proposition \ref{ss:prp:sctoss}, Proposition \ref{prp:ext_act_complexified} and the long exact sequence of homotopy groups, we see that we can apply Bott's theorem to $\Omega$, which, in connection to Theorem \ref{int:thm:homogeneous:converges_in_E_2}, gives us a complete algebraic description of $H^{p,q}(G; \frk h)$.
    This is exactly Theorem \ref{int:thm:homogeneous:homo_on_simply3}, which we can finally prove:
    

    \begin{proof}[Proof of Theorem \ref{int:thm:homogeneous:homo_on_simply3}]
        By Proposition \ref{ss:prp:sctoss}, we can assume $G$ to be simply-connected. Notice that $K$ is connected by construction, therefore, we use the long exact sequence of homotopy groups to conclude that $\Omega$ is simply-connected. Since $G$ is compact, we have that $\Omega$ is compact. By combining this with the result obtained in Proposition \ref{prp:ext_act_complexified}, we can apply Bott's theorem. The proof is completed.
    \end{proof}

    \begin{cor}
        Let $G$ be a semisimple compact Lie group and let $T \subset G$ be a maximal torus. Consider the following elliptic Lie algebra $ \frk h = \frk t \oplus \bigoplus_{\alpha \in \Delta_+} \frk g_\alpha$ with $\frk t$ being the complexification of the Lie algebra of the maximal torus $T$. Then we have $ H^{0,q}(G) = H^{0,q}(T).$
    \end{cor}
    
    \begin{proof}
    
        We define the homogeneous manifold $\Omega \doteq G / T$ and we denote the quotient map by $\pi : G \to \Omega$. Notice that we have precisely the hypothesis of Theorem \ref{int:thm:homogeneous:homo_on_simply3}. Therefore, we have
        $$ H^{0,q}(G) = \sum_{ r + s = q} H^{0,r}(T; \frk t) \otimes H^{0,s}_{\delbar}(\Omega) = H^{0,q}(T)$$ and the last equality follows from the fact that  $H^{p,q}_{\delbar}(\Omega) = 0$ if $p \neq q$ and $H^{0,0}_{\delbar}(\Omega) = \C$. 
    \end{proof}



\bibliographystyle{plain}
\bibliography{bibliography}

\printindex

\end{document}